\documentclass{article}

\usepackage{array}
\usepackage{amsfonts,amssymb,amsmath,amsthm}
\usepackage[all]{xy}
\usepackage{calrsfs}
\usepackage{hyperref}
\newtheorem{theorem}{Theorem}[section]
\newtheorem{proposition}[theorem]{Proposition}
\newtheorem{proposition-definition}[theorem]{Proposition-Definition}
\newtheorem{lemma}[theorem]{Lemma}
\newtheorem{corollary}[theorem]{Corollary}
\theoremstyle{remark}

\newtheorem{remark}[theorem]{Remark}

\theoremstyle{definition}

\newtheorem{definition}[theorem]{Definition}

\newcommand{\C}{\mathbb{C}}
\newcommand{\GL}{\operatorname{GL}}
\newcommand{\Rad}{\operatorname{Rad}}
\newcommand{\im}{\operatorname{Im}}
\newcommand{\K}{\mathbb{K}}

\newcommand{\Z}{\mathbb{Z}}
\newcommand{\rad}{\operatorname{rad}}
\newcommand{\N}{\mathbb{N}}
\newcommand{\image}{\operatorname{Im}}
\newcommand{\U}{\mathfrak{U}}
\newcommand{\dominate}{\ensuremath{\trianglerighteq}}
\newcommand{\wt}{\operatorname{wt}}

\newcommand{\Cat}{\mathcal{C}}
\newcommand{\A}{\mathfrak{A}}
\newcommand{\Cf}{\mathfrak{C}}
\newcommand{\charact}{\operatorname{char}}

\DeclareMathOperator{\End}{End}
\DeclareMathOperator{\Ext}{Ext}
\DeclareMathOperator{\Ker}{Ker}
\DeclareMathOperator{\Tor}{Tor}
\DeclareMathOperator{\Hom}{Hom}
\DeclareMathOperator{\Maps}{Maps}

\author{Ana Paula Santana and Ivan Yudin
\thanks{Financial support by CMUC/FCT gratefully acknowledged by both
authors.}
\thanks{The second author's work is supported by the FCT Grant
SFRH/BPD/31788/2006.}}
\title{The Kostant form of $\U(sl_n^+)$ and the Borel subalgebra of
the Schur algebra $S(n,r)$
}

\begin{document}
\maketitle
\begin{abstract}
 Let $\A_n(\K)$ be the Kostant form of $\U(sl_n^+)$ and $\Gamma$
 the monoid generated by the positive roots of $sl_n$. 
 For each $\lambda\in \Lambda(n,r)$ we construct a functor
 $F_{\lambda}$ from the category of finitely generated $\Gamma$-graded
 $\A_n(\K)$-modules to the category of finite-dimensional
 $S^+(n,r)$-modules, with the property that       $F_{\lambda}$ maps
 (minimal) projective resolutions of the one-dimensional
 $\A_n(\K)$-module $\K_{\A}$ to (minimal) projective resolutions of the simple $S^+(n,r)$-module
 $\K_{\lambda}$. 
\end{abstract}
\section*{Introduction}
The polynomial representations of the general linear group
$\GL_n(\C)$ were studied by I.~Schur in his doctoral
dissertation~\cite{schur1}. In this famous work, Schur introduced the,
now called, Schur algebras, which are a powerful tool to connect
$r$-homogeneous polynomial representations of the symmetric group on
$r$ symbols. 

These results of I.~Schur were generalised by J.-A.~Green to infinite
fields of arbitrary characteristic in~\cite{green}.  In Green's work
the Schur algebra $S(n,r)=S_{\K}(n,r)$ plays the central role in the
study of polynomial representations of $\GL_n(\K)$. 

In \cite{donkin1} Donkin shows that $S(n,r)$ is a quasi-hereditary
 and so it has finite global dimension.
This led to the problem of describing explicit
projective resolution of the Weyl modules for $S(n,r)$. Only partial
answers to this problem are known. In \cite{akin} and
\cite{z} such resolutions were constructed for the case when $\K$ is a field
of characteristic zero. If $\K$ has arbitrary characteristic then
projective resolutions of $W_{\lambda}$ are given in
\cite{akin-buchsbaum1}
when $n=2$ ($\lambda$ arbitrary), and in~\cite{maliakas} and~\cite{woodcock} for hook
partitions.    

In \cite{woodcock} Woodcock provides the tools to reduce the problem of
constructing these resolutions to the similar problem for the simple
modules for the Borel subalgebra $S^+(n,r)$ of $S(n,r)$.

Denote by
$\Lambda(n,r)$ the set of compositions of $r$ onto $n$ parts. It is
proved in~\cite{s} that all simple $S^+(n,r)$-modules are
one-dimensional and parametrised by the set $\Lambda(n,r)$. We denote
the simple module corresponding to $\lambda\in\Lambda(n,r)$ module by $\K_{\lambda}$.
In~\cite{s}, the first two
steps in a minimal projective resolution of $\K_{\lambda}$ and the first three terms of a minimal projective
resolution in the case $n=2$ are constructed. In~\cite{me1} minimal projective resolutions for
$\K_{\lambda}$ for $\lambda\in\Lambda(2,r)$ and
non-minimal projective resolutions of $\K_{\lambda}$ for
$\lambda\in\Lambda(3,r)$ are constructed. The results of both papers depend on      heavy calculations in the algebra $S^+(n,r)$. 

In the present paper we take a more abstract approach. 

Let us denote by $\A_n({\K})$ the Kostant form over the field $\K$ of the universal
enveloping algebra of the Lie algebra $sl^+_n$ of upper triangular
nilpotent matrices. Then $\A_n(\K)$ has a unique one-dimensional
module, which we denote by $\K_{\A}$. In this paper we show that the
construction of (minimal) projective resolutions for $K_{\lambda}$ is
essentially equivalent to the construction of (minimal) projective
resolution for $\K_{\A}$. The last task is much more feasible, since
an explicit presentation of $\A_n(\K)$ can be given   and thus the
results of Anick~\cite{anick} can be applied to the description of an
explicit projective resolution of $\K_{\A}$. It is also worth to
note that $\A_n(\K)$ is a projective limit of finite dimensional
algebras in the case $\charact(\K)=p>0$, and therefore the technique
developed in~\cite{butler-king} can be used for the construction of the
minimal projective resolution of $\K_{\A}$. This line of research will
be followed by us
 in the subsequent papers. 

The general plan of the present paper is as follows. In
Section~\ref{abstract} we collect general technical results, which
will be used in the following section. 
We believe that these results can be applied in more general context,
in particular to the generalised and $q$-Schur algebras. 

Let $G$ be an ordered group with neutral element $\epsilon$.
Denote by $\Gamma\subset G$ the submonoid of non-negative elements of
$G$. For every $\Gamma$-graded algebra $A$ and every $\Gamma$-set
$S$, we construct a family of $\Gamma$-algebras
$$
\left\{ C(X)\middle| X\subset S \right\}
$$
and a family of $\Gamma$-graded algebra homomorphisms
$$
\left\{ \phi^Y_X\colon C(Y)\to C(X)\middle| X\subset Y \subset S \right\}
$$
such that $\phi^Y_X\circ \phi^Z_Y = \phi^Z_X$ for every triple
$X\subset Y \subset Z$ of subsets in $S$. In other words, $C(-)$ is a
presheaf of $\Gamma$-graded algebras.

For every $x\in S$, we construct an exact functor $F_x$ from the
category $\Cat(A,\Gamma)$ of finitely generated $\Gamma$-graded
$A$-modules to the category $\Cat(C(S),\Gamma)$. If $\Gamma$ acts by
automorphisms on $S$, then the functors $F_x$ preserve projective
modules, and thus map projective resolutions into projective
resolutions. If $A_{\epsilon}\cong \K$, then the functors $F_x$ map
minimal projective resolutions into minimal projective resolutions. 

For every $X\subset S$, we can consider each left $C(X)$-module as a
left $C(S)$-module via a homomorphism $\phi^S_X\colon C(S)\to
C(X)$. Thus we get a natural inclusion of categories
$$
(\phi^S_X)^*\colon \Cat(C(X),\Gamma) \to \Cat(C(S),\Gamma).
$$
There is a left adjoint functor to $\left( \phi^S_X \right)^*$
$$
(\phi^S_X)_*= C(X)\otimes_{C(S)}- \colon \Cat(C(S),\Gamma) \to \Cat(C(X),\Gamma),
$$
where we consider $C(X)$ as a right $C(S)$-module via $\phi^S_X$. The
main objective of Section~\ref{criterium} is to get conditions on
$X\subset S$, ensuring that for every left $C(X)$-module $M$ and
every (minimal) projective resolution $P_{\bullet}$ of $\left(\phi^S_X
\right)^*(M)$ the complex $\left(\phi^S_X \right)_*(P_{\bullet})$ is a
(minimal) projective resolution of $M \cong \left( \phi^S_X
\right)_*\left( \phi^S_X \right)^*(M)$. If the both algebras $C(S)$ and
$C(X)$ are artinian and $\left( \phi^S_X \right)_*$ has the above
mentioned property,
then the ideal $\Ker\left( \phi^S_X \right)$ is a \emph{strong
idempotent ideal}. The algebra $C(X)$ is finite dimensional and thus
artinian, if
$X$ is finite and $A$ is locally finite dimensional. But the algebra
$C(S)$ is rarely finite dimensional. To cope with this, we take a two
stage approach. 

We say that $Y\subset S$ is $\Gamma$-convex, if from
$\gamma=\gamma_1\gamma_2$ and $x$, $\gamma x\in Y$ it follows that
$\gamma_2 x\in Y$. In Proposition~\ref{split}, we show that if
$Y$ is a convex $\Gamma$-set then the
functor $\left( \phi^S_Y \right)_*$ is exact and maps (minimal)
projective resolutions into (minimal) projective resolutions. 

Let $Y$ be a finite $\Gamma$-convex subset of $S$ and $X$ a subset of
$Y$. Suppose that $A$ is
locally finite dimensional. In Theorem~\ref{induction}  we give a
criterion for $\Ker\left( \phi^Y_X \right)$ to be a strong idempotent
ideal. 

In Section~\ref{application} we apply the results of  
Section~\ref{abstract}
to $S^+(n,r)$. In our particular case, the algebra $A$ is the Kostant form
$\A_n(\K)$, the  set $S$ is      $\Z^n$ and $\Gamma$ is the submonoid
of $\Z^n$ generated by the elements $(0,\dots,1,-1,\dots,0)$. Then we
show in Theorem~\ref{isomorphismprime} that $C(\Lambda(n,r))\cong
S^+(n,r)$. Here we consider $\Lambda(n,r)$ as a subset $\Z^n$ in the
natural way. Note that this isomorphism gives a description of
$S^+(n,r)$, which is similar to the idempotent presentation of the
algebra $S(n,r)$ obtained by Doty and Giaquinto in~\cite{doty}. 

The set of compositions $\Lambda(n,r)$ is contained in the lager finite
set $\Lambda^1(n,r)$ defined by 
$$
\Lambda^1(n,r) = \left\{ (z_1,z_2,\dots,z_n)\in \Z^n\middle| 0\le
\sum_{i=1}^j z_i,\,1\le j\le r-1; \sum_{i=1}^n z_i = r \right\}.
$$
It turns out, that $\Lambda^1(n,r)$ is a $\Gamma$-convex set
(see~Proposition~\ref{dominance}) and therefore the functor $\left(
\phi^{\Z^n}_{\Lambda^1(n,r)} \right)_*$ is exact and preserves
(minimal) projective resolutions. Moreover, we show in
Theorem~\ref{main1} that 
$\Ker\left( \phi^{\Lambda^1(n,r)}_{\Lambda(n,r)} \right)$ is a
strong idempotent ideal. Hence the composite functor
$$
\left( \phi^{\Lambda^1(n,r)}_{\Lambda(n,r)})\right)_*\circ \left(
\phi^{\Z^n}_{\Lambda^1(n,r)} \right)_* = \left(
\phi^{\Z^n}_{\Lambda(n,r)} \right)_*
$$
preserves projective resolutions of $C(\Lambda(n,r))=S^+(n,r)$-modules
considered as $C(\Z^n)$-modules.

Recall, that $\K_{\A}$ denotes the
unique one-dimensional module over $\A_n(\K)$ and $F_{\lambda}$ is the
functor from $\Cat(\A_n(\K),\Gamma)$ to $\Cat(C(\Z^n),\Gamma)$
associated with $\lambda\in \Z^n$. It follows from the definitions, that
$$
F_{\lambda}(\K_{\A}) \cong \left( \phi^{\Z^n}_{\Lambda(n,r)}
\right)^*(\K_{\lambda}). 
$$
Therefore, if $P_{\bullet}$ is a (minimal) projective resolution of
$\K_{\A}$, then 
$$
\left( \phi^{\Z^n}_{\Lambda(n,r)} \right)_*\circ
F_{\lambda}(P_{\bullet})
$$
is a (minimal) projective resolution of $\K_{\lambda}$.

Now we give a more detailed outline of the paper. In
Section~\ref{eilenberg}, we make a (fairly simple) extension of
Eilenberg machinery on perfect categories to the case of
$\Gamma$-graded algebras. Most results are valid only for
positive monoids, that is for submonoids of non-negative elements
in an ordered group. 

In Section~\ref{skew} we define the  skew product  $A\ltimes_{\Gamma} B$ of a
$\Gamma$-graded algebra $A$ and $\Gamma$-algebra $B$ over the monoid
$\Gamma$. If $\Gamma$ is a group $G$ and $A$ is the group algebra
$\K[G]$, then $\K[G]\ltimes_{G} B$ is isomorphic to the usual skew
product of $B$ with $G$. For every $B$-module $N$ we construct an
exact functor 
$$
-\ltimes_{\Gamma} N\colon \Cat(A,\Gamma) \to \Cat(A\ltimes_{\Gamma} B,
\Gamma),
$$
and establish conditions for $-\ltimes_{\Gamma} N$ to preserve
(minimal) projective resolutions. 

Section~\ref{strong-idempotent} is an      overview of results of  
homological algebra, which we use after. In particular, we recall the
notions of strong idempotent ideal and of heredity ideal and some of
their properties. 

Section~\ref{criterium} is the central section of the first part of
our work. Here we
prove a criterion for $\Ker\left( \phi^Y_X \right)$ to be a strong
idempotent ideal.

In Section~\ref{combinatorics} we prove   the results about
compositions, multi-indices and orderings on $\Z^n$, which we use
later on. 

The Schur algebra $S(n,r)$ and its Borel subalgebra $S^+(n,r)$ as well
as the algebra $\A_n(\K)$ are considered in Sections~\ref{schur}
and~\ref{kostant}, respectively.

Finally, in Section~\ref{main} we prove that $\Ker\left(
\phi^{\Lambda^1(n,r)}_{\Lambda(n,r)} \right)$ is the strong idempotent
ideal and that $C(\Lambda(n,r))\cong S^+(n,r)$.

\section{Skew product over monoids and strong idempotent ideals}
\label{abstract}
\subsection{Graded rings and modules}
\label{eilenberg}
In this subsection we recollect results about graded algebras and
graded modules, that were essentially proved in Eilenberg's
paper~\cite{eilenbergMR}.

Let $\Gamma$ be a monoid with neutral element $\epsilon$ and $A$ a
$\Gamma$-graded associative algebra, that is
$$
A = \bigoplus\limits_{\gamma\in \Gamma} A_{\gamma},
$$
where $A_{\gamma}$ is a subspace of $A$, for each $\gamma\in \Gamma$,
and if $a_1\in A_{\gamma_1}$ and $a_2\in A_{\gamma_2}$ then $a_1a_2
\in A_{\gamma_1\gamma_2}$.
We will assume in addition that the unity $e_A$ of $A$ is an element
of $A_{\epsilon}$. 

A left $A$-module $M$ is $\Gamma$-graded if $M=
\bigoplus\limits_{\gamma\in \Gamma} M_{\gamma}$, where each
$M_{\gamma}$ is a vector subspace of $M$, and for $a\in
A_{\gamma_1}$ and $m\in M_{\gamma_2}$ we have $am\in
M_{\gamma_1\gamma_2}$.

We will consider the category $A$-$\Gamma$-gr of left $\Gamma$-graded
$A$-modules and $A$-module homomorphisms respecting                 
    grading, that is a map of $A$-modules $f\colon M_1\to M_2$ is in
$A$-$\Gamma$-gr if $f(M_{1,\gamma})\subset M_{2,\gamma}$ for all
$\gamma\in \Gamma$.
Define the radical $\rad(A)$ of $A$ as the intersection of all
maximal graded left ideals of $A$. We also use the notion of ungraded radical
$\Rad(A)$ of $A$, which is defined as the intersection of all
(not-necessarily graded) maximal left ideals of $A$. 
In the following, we determine  the conditions under which
$\rad(A)$ and $\Rad(A)$ coincide. 
\begin{definition}
We say that the monoid $\Gamma$ is positive, if it has the property
$$
\gamma_1\gamma_2 = {\epsilon} \Rightarrow \gamma_1={\epsilon} \mbox{
and }\gamma_2={\epsilon}.
$$
\end{definition}

\begin{lemma}
 \label{maximal_graded}
 Let $\Gamma$ be a positive monoid and $A$ a $\Gamma$-graded
 algebra. If $m$ is a maximal graded left ideal of $A$, then
 $$
 m = m_{\epsilon} \oplus \bigoplus_{\gamma\ne {\epsilon}} A_{\gamma}
 $$
 where $m_{\epsilon}$ is a maximal left ideal of $A_{\epsilon}$.
\end{lemma}
\begin{proof}
 Let $n$ be a maximal left ideal of $A_{\epsilon}$ containing
 $m_{\epsilon}$. It is clear, that 
 $$
 n' = n \oplus \bigoplus_{\gamma\ne {\epsilon}} A_{\gamma}
 $$
 is a $\Gamma$-graded left ideal of $A$ and that $m\subset n'$. Since $m$ is
 maximal, it follows that $m = n'$. In particular, $n = m_{\epsilon}$.  
\end{proof}

\begin{corollary}
 \label{subset}
Suppose $\Gamma$ is positive. Then
$\bigoplus_{\gamma\ne
{\epsilon}} A_{\gamma}$ is a subset of $\rad(A)$. 
\end{corollary}
\begin{proof}
 By Lemma~\ref{maximal_graded} $\bigoplus_{\gamma\ne {\epsilon}}A_{\gamma}$ is
 a subset of each maximal graded left ideal of $A$. 
\end{proof}

\begin{corollary}
 \label{ungraded} 
 Let $\Gamma$ be a positive monoid and $A$ a $\Gamma$-graded
 algebra. If $m$ is a maximal graded left ideal of $A$, then $m$ is a
 maximal left ideal of $A$ in the ungraded sense.  
\end{corollary}
\begin{proof}
 Since $\bigoplus_{\gamma\ne
{\epsilon}} A_{\gamma}$ is an ideal of $A$, we have a surjective homomorphism
$A\to A_{\epsilon}$. As    $m_{\epsilon}$ is a maximal ideal of
$A_{\epsilon}$, the
left  $A_{\epsilon}$-module $A_{\epsilon}/m_{\epsilon}$ is simple. But
then $A_{\epsilon}/m_{\epsilon}$ is
simple as an $A$-module. Since $A_{\epsilon}/m_{\epsilon}\cong A/m$, we get that
$m$ is a maximal left ideal of $A$. 
\end{proof}
\begin{corollary}
 \label{structureofradical} Suppose $\Gamma$ is positive, then
 $\rad(A) = \Rad(A_{\epsilon}) \oplus
 \bigoplus_{\gamma\ne {\epsilon}} A_{\gamma}$. 
\end{corollary}
\begin{proof}
 The radical $\rad(A)$ is a graded ideal of $A$, thus
 $$
 \rad(A) = \bigoplus_{\gamma \in \Gamma}\rad(A)_{\gamma}.
 $$
 Since by Corollary~\ref{subset} $\bigoplus_{\gamma\ne {\epsilon}}A_{\gamma}$ is a subset of
 $\rad(A)$ we have that $\rad(A)_{\gamma} = A_{\gamma}$ for $\gamma\ne
 {\epsilon}$. Thus it is enough to check that $\rad(A)_{\epsilon}=
 \Rad(A_{\epsilon})$. 

 Let $m$ be a maximal left ideal of $A_{\epsilon}$. Then 
 $$
 m'= m\oplus \bigoplus_{\gamma\ne {\epsilon}} A_{\gamma}
 $$
 is a maximal graded left ideal of $A$. Now
 $$
 \rad(A)_{\epsilon} = \rad(A)\cap A_{\epsilon} \subset m'\cap A_{\epsilon} = m.
 $$ 
 As    $m$ was an arbitrary maximal left ideal of $A_{\epsilon}$ and
 $\Rad(A_{\epsilon})$ is
 the intersection of all such left ideals, we get that $\rad(A)_{\epsilon}\subset
 \Rad(A_{\epsilon})$. 

 Now, let $m$ be a maximal graded left ideal of $A$. Then 
 by Lemma~\ref{maximal_graded}
 $$
 m= m_{\epsilon} \oplus \bigoplus_{\gamma\ne {\epsilon}}A_{\gamma},  
 $$
 where $m_{\epsilon}$ is a maximal left ideal of $A_{\epsilon}$. Therefore 
 $$
 \Rad(A_{\epsilon}) \subset m_{\epsilon} \subset m
 $$
 and since $m$ was an arbitrary maximal graded left ideal of $A$
 $$
 \Rad(A_{\epsilon}) \subset \rad(A) \cap A_{\epsilon} = \rad(A)_{\epsilon}. 
 $$
\end{proof}
For $\gamma\in \Gamma$ we denote by $A(\gamma)$ the left $\Gamma$-graded
module, defined by
$$
A(\gamma)_{\gamma'} := A_{\gamma'\gamma}
$$
and the action of $A$ is given by multiplication:
\begin{align*}
 A_{\gamma_1}\otimes A(\gamma)_{\gamma_2} & \to
 A(\gamma)_{\gamma_1\gamma_2}\\
 a_1\otimes a_2&\mapsto a_1a_2.
\end{align*}
\begin{definition}
 We call a direct sum of modules of the form $A(\gamma)$ \emph{a free}
$\Gamma$-graded $A$-module. A direct summand in the category of
$\Gamma$-graded $A$-modules of a free $\Gamma$-graded $A$-module is
called \emph{a projective $\Gamma$-graded $A$-module}.\end{definition}

Let $P$ be a projective $\Gamma$-graded $A$-module generated by a
single element $x\ne 0$ of degree $\gamma\in \Gamma$. Let $F$ be a
free $\Gamma$-graded module with a basis consisting of an element
$y$ of degree $\gamma$. Then there is an epimorphism $\phi\colon F\to
P$, such that $\phi y = x$. Since $P$ is projective, it follows that
$\Ker(\phi)$ is a direct summand of $F$. Consequently there exists an
idempotent $e\in A_{\epsilon}$ such that $P\cong Aex$. The degree
$\gamma$ of $x$ is uniquely determined by $P$. 

\begin{definition}
 A left $\Gamma$-graded $A$-module is called \emph{quasi-free} if it
 is a direct sum of projective modules each of which is generated by
 a single (homogeneous) element. 
\end{definition}

In the following we assume, that 
\begin{itemize}
 \item the monoid $\Gamma$ is positive;
 \item the ring $\bar{A} = A/ \rad(A) = A_{\epsilon} / \Rad(A_{\epsilon})$  is
 semi-simple;
 \item each idempotent in $\bar{A}$ can be lifted to $A_{\epsilon}$. 
\end{itemize}

Following Eilenberg~\cite{eilenbergMR} we say that the subcategory
$\Cat$ of $A$-$\Gamma$-gr is \emph{perfect} if 
\begin{enumerate}
 \item $\Cat$ is full.
 \item If $\phi\colon M\to N$ is an epimorphism and $M\in \Cat$ then
 $N\in \Cat$. 
 \item $A\in \Cat$.
 \item If $P$ is quasi-free and $P/\rad(A)P\in \Cat$, then $P\in \Cat$. 
 \item If $M\in \Cat$ and $\rad(A)M = M$, then $M=0$.  
\end{enumerate}

Suppose $\Cat$ is a perfect subcategory of $A$-$\Gamma$-gr.
\begin{definition}
 An epimorphism $\phi\colon P \to M$ in $\Cat$ is called
 \emph{minimal} if $P$ is projective and $\Ker(\phi)\subset
 \rad(A) P$. 
\end{definition}

\begin{proposition}
 \label{minimalepimorphism}
 Every $M\in \Cat$ admits a minimal epimorphism $\phi\colon P\to M$.
 If $\phi'\colon P'\to M$ is another minimal epimorphism, then there
 exists a homomorphism $\pi\colon P\to P'$ such that $\phi'\pi = \phi$,
 and each such homomorphism is an isomorphism. 
\end{proposition}
\begin{proof}
 The proof is word by word repetition of the proof
 of~\cite[Proposition~3]{eilenbergMR}, with understanding that all
 $\N$-graded modules have to be replaced by $\Gamma$-graded modules. 
\end{proof}

Suppose further, that $\Cat$ satisfies the additional condition
\begin{itemize}
 \item If $M\in \Cat$ and $N\subset M$, then $N\in \Cat$. 
\end{itemize}

Then, as usual, by iterating the minimal epimorphism construction we
get for each $M\in \Cat$ a projective resolution 
$$
\dots \to P_n \stackrel{d_n}{\to} P_{n-1} \to \dots
\stackrel{d_1}{\to} P_0 \stackrel{\varepsilon}{\to} M \to 0,
$$
such that $\im(d_i) \subset \rad(A)P_i$ for $i=0,1,\dots$. This
projective resolution is called \emph{minimal}. 

The following two results are formal consequences of the definition
and Proposition~\ref{minimalepimorphism}. For proofs the reader is
referred to~\cite[Proposition~7 and Theorem~8]{eilenbergMR}. 
\begin{proposition}
 Let $P_{\bullet}$ and $P'_{\bullet}$ be minimal projective resolutions of a
 module $M\in \Cat$. Then there exists a map $f\colon P_{\bullet} \to
 P'_{\bullet}$ over the identity map of $M$ and each such map is an
 isomorphism. 
\end{proposition}
\begin{theorem}
 Let $M\in \Cat$ and let $P_{\bullet}$ be a projective resolution of $M$. Then
 $P_{\bullet}$ decomposes into a direct sum $P_{\bullet} =
 \bar{P}_{\bullet} \oplus W_{\bullet}$ of
 subcomplexes such that $\bar{P}_{\bullet}$ is a minimal projective resolution of
 $M$, while $W$ is a projective resolution of the zero module. 
\end{theorem}

From now on      we restrict ourselves to the case when $\Gamma$ is
an ordered monoid and the neutral element ${\epsilon}$
of $\Gamma$ is the least   element of $\Gamma$. 
Note that all such monoids are positive, since
$$
\gamma_1\ne \epsilon,\ \gamma_2\ne \epsilon \Rightarrow
\gamma_1>\epsilon,\ \gamma_2>\epsilon \Rightarrow \gamma_1\gamma_2>
\epsilon \Rightarrow \gamma_1\gamma_2 \ne \epsilon.
$$

A $\Gamma$-graded $A$-module $M$ is said to be \emph{locally finitely
generated} if $M$ is generated by a set $X$ of (homogeneous) elements
such that the sets $X\cap M_{\gamma}$ are finite, for all
$\gamma\in\Gamma$. 

\begin{theorem}
 \label{perfectcategory} 
 The category $\Cat(\Gamma,A)$ of all locally finitely generated
 $\Gamma$-graded left $A$-modules is a perfect subcategory of
 $A$-$\Gamma$-gr. It is closed under taking subobjects if and only if
 $A_{\epsilon}$ is a left Noetherian ring and each $A_{\gamma}$ is
 finitely generated as a left $A_{\epsilon}$-module. 
\end{theorem} 

\begin{proposition}
 \label{finitedimensional}
Let $A$ be a finite dimensional $\Gamma$-graded algebra.
Then the ungraded radical $\Rad(A)$ of $A$ coincides
with the graded radical $\rad(A)$ of $A$.
\end{proposition}
\begin{proof}
 From Corollary~\ref{ungraded} it follows that $\Rad(A)$ is a subset
 of $\rad(A)$.

  We shall show, that $\rad(A)$ is nilpotent. Since
 $\Rad(A)$ is a maximal nilpotent ideal of $A$, we shall get that
 $\rad(A) = \Rad(A)$.

 By Corollary~\ref{structureofradical}, we have $\rad(A) =
 \Rad(A_e) \oplus \bigoplus_{\gamma\ne e} A_{\gamma}$.
 Denote by $N$ the ideal $\bigoplus_{\gamma\ne e}A_{\gamma}$.
 Then $N$ is nilpotent.
In fact, we show, that if $\dim(A) =n$, then the product of any $n+1$ elements of  $N$ is
 zero. Clearly, this should be checked only for homogeneous elements.
 Let $a_0, a_1, \dots a_n$ be  a sequence of homogeneous
 elements from $N$. Suppose, that for each $i$ the element $a_i$ is
 from $A_{\gamma_i}$. Then, since each $\gamma_i> \epsilon$ we have a
 strictly 
 increasing sequence 
 $$
 \gamma_0< \gamma_0\gamma_1< \dots< \gamma_0\gamma_1\dots \gamma_n.
 $$
 As    $A$ is $n$-dimensional
 one of the $n+1$ spaces 
 $$
 A_{\gamma_0}, A_{\gamma_0\gamma_1}, \dots ,
 A_{\gamma_0\gamma_1\dots\gamma_n}
 $$
 should be zero. In particular, one of the products
 $$
 a_{\gamma_0}, a_{\gamma_0}a_{\gamma_1}, \dots,
 a_{\gamma_0}a_{\gamma_1}\dots a_{\gamma_n}
 $$
 is zero. Thus $a_{\gamma_0}a_{\gamma_1}\dots a_{\gamma_n} = 0$
 and
 $N^{n+1} = 0$. 

 Denote $\Rad(A_{\epsilon})$ by $M$.
 For any natural numbers $k_0,\dots, k_{n+1}$, we have
 \begin{multline*}
  M^{k_0}NM^{k_1}\dots M^{k_{n}} N
  M^{k_{n+1}}=\\
  \left( M^{k_0}N \right)\left( M^{k_1}N \right)\dots
  \left( M^{k_{n-1}}N \right) \left( M^{k_{n}} N
    M^{k_{n+1}}\right)\subset N^{n+1} = 0.
 \end{multline*}

 As    $\Rad(A_{\epsilon})$ is a nilpotent ideal of the algebra
 $A_{\epsilon}$, there is
 $m$, such that $(\Rad(A_{\epsilon}))^m = M^m = 0$. 
  Then 
$$
(M + N)^{nm + n} = \sum_{l=0}^{n-1}\sum_{(k_0,\dots,k_l):\sum
k_i = nm + n
-l} M^{k_0}NM^{k_1}\dots M^{k_{l-1}} N
M^{k_l} = 0,
$$
since in each summand at least one $k_i$ is greater than $m$.
\end{proof}
\subsection{Skew product over monoids}
\label{skew}
We say that an algebra $B$ is a $\Gamma$-algebra, if there is a given
right action
\begin{align*}
 r\colon B\times \Gamma&\to B\\
 (b,\gamma)  & \mapsto b^{\gamma}
\end{align*}
 such that for each
$\gamma\in \Gamma$ the map 
\begin{align*}
  B &\to B\\
  b & \mapsto b^{\gamma}
\end{align*}
is an algebra homomorphism. 

Let $A$ be a $\Gamma$-graded algebra.
We define the interchange  map $T\colon B\otimes A\to A\otimes B$ by
\begin{align*}
 B\otimes A_{\gamma} & \to A_{\gamma} \otimes B\\
b\otimes a_{\gamma}&  \mapsto a_{\gamma}\otimes b^{\gamma}
\end{align*}
and a binary operation $m$ on $A\otimes B$ by
$$
m\colon A\otimes B \otimes A\otimes B \xrightarrow{1_A\otimes T\otimes
1_B} A\otimes A\otimes B\otimes B
\xrightarrow{m_A\otimes m_B} A\otimes B.
$$
Denote the vector space $A\otimes B$ with the binary operation $m$ on it by
$A\ltimes_{\Gamma} B$. 
\begin{proposition}
 $A\ltimes_{\Gamma}B$ is an algebra.
\end{proposition}
\begin{proof}
 We have to check that $e_A\otimes e_B$ is a neutral element with
 respect to $m$ and that $m$ is an associative operation. This follows
 from the following computation
 $$
 e_A\otimes e_B \otimes a_{\gamma} \otimes b \mapsto e_A\otimes
 a_{\gamma} \otimes e_B^{\gamma} \otimes b \mapsto a_{\gamma} \otimes
 b,
 $$
 $$
 a_{\gamma}\otimes b\otimes e_A\otimes e_B\mapsto a_{\gamma}\otimes
 e_A\otimes b^e \otimes e_B \mapsto a_{\gamma}\otimes b
 $$
 and 
 $$
 \xymatrix{
 {a_{\gamma_1}\otimes b_1\otimes a_{\gamma_2}\otimes b_2 \otimes
 a_{\gamma_3}\otimes b_3} \ar@{|->}[r]\ar@{|->}[dd] &
 { a_{\gamma_1}a_{\gamma_2}\otimes b_1^{\gamma_2}b_2 \otimes
 a_{\gamma_3}\otimes b_3}\ar@{|->}[d]\\
 &{a_{\gamma_1}a_{\gamma_2}a_{\gamma_3}\otimes
 (b_1^{\gamma_2}b_2)^{\gamma_3}b_3}\ar@{=}[d]\\
 {a_{\gamma_1}\otimes b_1\otimes a_{\gamma_2}a_{\gamma_3}\otimes
 b_2^{\gamma_3}b_3}
 \ar@{|->}[r] & {a_{\gamma_1}a_{\gamma_2}a_{\gamma_3}\otimes
 b_1^{\gamma_2\gamma_3}b_2^{\gamma_3}b_3  .}
 } 
 $$
\end{proof}
Note, that the embedding $A\to A\ltimes_{\Gamma} B$ given by
$$
a\mapsto a\otimes 1_B
$$
is a homomorphism of algebras. In the following we will consider the
elements of $A$ as elements of $A\ltimes_{\Gamma} B$ through this
embedding.

The algebra $A\ltimes_{\Gamma}B$ is itself $\Gamma$-graded.
In fact $$A\ltimes_{\Gamma} B = \bigoplus_{\gamma\in
\Gamma}\left(A_{\gamma}\otimes B\right).$$
Let $N$ be a $B$-module and $M= \bigoplus_{\gamma\in
\Gamma}M_{\gamma}$ a $\Gamma$-graded $A$-module. We
define a $(A\ltimes_{\Gamma}B)$-module structure on $M\otimes N$ as
follows
$$
a_{\gamma_1}\otimes b \otimes m_{\gamma_2}\otimes n \mapsto
a_{\gamma_1}m_{\gamma_2}\otimes b^{\gamma_2}n
$$
for all $a_{\gamma_1} \in A_{\gamma_1}$, $b\in B$, $m_{\gamma_2}\in
M_{\gamma_2}$ and $n\in N$. We denote this module by
$M\ltimes_{\Gamma}N$.
This is a $\Gamma$-graded module, with $(M\ltimes_{\Gamma}
N)_{\gamma} = M_{\gamma}\otimes N$.

 Let $\varphi\colon M_1\to M_2$ be a homomorphism of $\Gamma$-graded $A$-modules
 and $\psi\colon N_1\to N_2$ a homomorphism of $B$-modules. We denote by
 $\varphi\ltimes_{\Gamma}\psi$ the map
 \begin{align*}
  M_1\ltimes_{\Gamma}N_1& \to M_2\ltimes_{\Gamma}N_2\\
  m\otimes n  & \mapsto \varphi(m)\otimes \psi(n).
 \end{align*}
 \begin{proposition}
  \label{homomorphism}
  The map $\varphi\ltimes_{\Gamma} \psi$ is a homomorphism of
  $\Gamma$-graded $A\ltimes_{\Gamma} B$-modules. 
\end{proposition} 
\begin{proof}
 Since for all $m_{\gamma} \in (M_1)_{\gamma}$, the element
 $\varphi(m_{\gamma})$ is an element of $(M_2)_{\gamma}$, it follows
 that $\varphi\ltimes_{\Gamma} \psi(m_{\gamma}\otimes n)=
 \varphi(m_{\gamma})\otimes \psi(n)$ is an
 element of $(M_2)_{\gamma}\otimes N$. Thus
 $\varphi\ltimes_{\Gamma}\psi$ preserves the $\Gamma$-grading.

 That $\varphi\ltimes_{\Gamma} \psi$ is a
 $A\ltimes_{\Gamma}B$-homomorphism follows from the following
 computation 
$$
\xymatrix{ {a_{\gamma_1}\otimes b \otimes m_{\gamma_2}\otimes n}
\ar@{|->}[r]\ar@{|->}[dd] & {a_{\gamma_1}m_{\gamma_2}\otimes b^{\gamma_2}n}\ar[d]\\
& \phi(a_{\gamma_1}m_{\gamma_2})\otimes \psi(b^{\gamma_2}n)\ar@{=}[d]\\
{a_{\gamma_1}\otimes b \otimes \varphi(m_{\gamma_2})\otimes \psi(n)} \ar@{|->}[r]&
{a_{\gamma_1}\varphi(m_{\gamma_2})\otimes b^{\gamma_2}\psi(n),}
 }
 $$
 where $a_{\gamma_1} \in A_{\gamma_1}$ and $m_{\gamma_2}\in
 M_{\gamma_2}$. 
\end{proof}
It follows from these results that the correspondence
\begin{align*}
 (M,N)& \mapsto M\ltimes_{\Gamma}N\\
 (\varphi,\psi)& \mapsto \varphi\ltimes_{\Gamma} \psi
\end{align*}
gives a bifunctor from the categories $A$-$\Gamma$-gr and $B$-mod to the
category $(A\ltimes_{\Gamma}B)$-$\Gamma$-gr.
In particular, for each $B$-module $N$ we have the functor $-\ltimes_{\Gamma}N$
from the category $A$-$\Gamma$-gr to the category
$(A\ltimes_{\Gamma}B)$-$\Gamma$-gr. This functor is obviously exact,
since it is just a tensor product with $N$ on the level of underlying
vector spaces.

\begin{proposition}
 \label{free}
Suppose, that $\Gamma$ acts by automorphisms on $B$. Let $F_A$ be a
$\Gamma$-graded free $A$-module, and $F_B$ a
free $B$-module. Then the module $F_A\ltimes_{\Gamma} F_B$ is a
$\Gamma$-graded
free $A\ltimes_{\Gamma} B$-module.
\end{proposition}
\begin{proof}
  Let $\left\{ v_{\alpha}\mid \alpha\in I \right\}$ be a
 $\Gamma$-homogeneous $A$-basis of
 $F_A$ and $\left\{ w_{\beta}\mid \beta\in J \right\}$ a $B$-basis of
 $F_B$. We shall show, that $\left\{ v_{\alpha}\otimes w_{\beta}\mid \alpha\in I,\
 \beta\in J \right\}$ is a $\Gamma$-homogeneous $A\ltimes_{\Gamma}B$-basis of
 $F_A\ltimes_{\Gamma} F_B$.

 First we show that each element in $F_A\ltimes_{\Gamma} F_B$ can be
 written as a $A\ltimes_{\Gamma} B$-combination of elements from
 $\left\{ v_{\alpha}\otimes w_{\beta}\mid \alpha\in I,\
 \beta\in J \right\}$. Clearly, it should be
 checked only for elements of the form $u\otimes v$ with $u\in A$ and
 $v\in B$. Since $\left\{ v_{\alpha}\mid \alpha\in I \right\}$ is a
 basis of $A$ we can write $u$ as a linear combination
 $$
 u= \sum_{\alpha\in I} x_{\alpha}v_{\alpha},
 $$
 with $x_{\alpha}\in A$. Since $\left\{ w_{\beta}\mid \beta\in J \right\}$ 
 is a basis of $B$ we can write $v$ as a linear combination
 $$
 v = \sum_{\beta\in J} y_{\beta}w_{\beta},
 $$
 with $y_{\beta}\in B$. Denote by $\gamma_{\alpha}$ degree of
 $v_{\alpha}$. Recall, that $\Gamma$ acts by automorphisms on
 $B$ and therefore the map $\gamma^{-1}\colon B\to B$ is well defined
 for all $\gamma\in \Gamma$. Then
 $$
 u\otimes v = \sum_{\alpha\in I} \sum_{\beta\in B}
 x_{\alpha}v_{\alpha}\otimes y_{\beta}w_{\beta}= \sum_{\alpha\in I}\sum_{\beta\in J}
 (x_{\alpha}\otimes \gamma_{\alpha}^{-1}(y_{\beta})
 )(v_{\alpha}\otimes w_{\beta}).
 $$

 Now, we show that the set $\left\{ v_{\alpha}\otimes w_{\beta}\mid \alpha\in I,\
 \beta\in J \right\}$ is linearly independent over $A\ltimes_{\Gamma}
 B$.
 Let $\left\{ a_{\theta} \right\}$ be a homogeneous $\K$-basis of
 $A$, and $\left\{ b_{\mu} \right\}$ a $\K$-basis of $B$. 
 Since $\Gamma$ acts by automorphisms on $B$, for each $\gamma\in
 \Gamma$ there are elements $b_{\mu,\gamma}$ such that
 $\gamma(b_{\mu,\gamma})=b_{\mu}$ and the set $\left\{ b_{\mu,\gamma}
 \right\}$ is a $\K$-basis of $B$. Thererfore, for each $\gamma\in
 \Gamma$ the set $\left\{ a_{\theta}\otimes b_{\mu,\gamma} \right\}$ is a
 $\K$-basis of $A\ltimes_{\Gamma}B$.
 Suppose, that
 $$
 \sum_{\alpha\in I}\sum_{\beta\in J}
 \kappa_{\alpha,\beta}v_{\alpha}\otimes w_{\beta}=0,
 $$
 where $\kappa_{\alpha,\beta}\in A\ltimes_{\Gamma}B$. Then there are
 elements $\eta_{\alpha,\beta,\theta,\mu}$ of $\K$ such that
 $$
 \kappa_{\alpha,\beta} = \sum_{\theta}\sum_{\mu}
 \eta_{\alpha,\beta,\theta,\mu} a_{\theta}\otimes
 b_{\mu,\gamma_{\alpha}}.
 $$
Thus
\begin{multline*}
0= \sum_{\alpha\in I}\sum_{\beta\in J}
 \kappa_{\alpha,\beta}v_{\alpha}\otimes w_{\beta}=\sum_{\alpha\in I,
 \beta\in J}\sum_{\theta,\mu}
 \eta_{\alpha,\beta,\theta,\mu}(a_{\theta}\otimes
 b_{\mu,\gamma_{\alpha}})(v_{\alpha}\otimes w_{\beta})\\ = \sum_{\alpha\in I,
 \beta\in J}\sum_{\theta,\mu}
 \eta_{\alpha,\beta,\theta,\mu}a_{\theta}v_{\alpha}\otimes
 \gamma_{\alpha}( b_{\mu,\gamma_{\alpha}}) w_{\beta} = 
\sum_{\alpha\in I,
 \beta\in J}\sum_{\theta,\mu}
 \eta_{\alpha,\beta,\theta,\mu}a_{\theta}v_{\alpha}\otimes
 b_{\mu} w_{\beta}.
\end{multline*}
Since $\left\{ a_{\theta}v_{\alpha} \right\}$ is a $\K$-basis of
$F_{A}$ and $\left\{ b_{\mu}w_{\beta} \right\}$ is a $\K$-basis of
$F_{B}$, it follows, that $\left\{ a_{\theta}v_{\alpha} \otimes
b_{\mu}w_{\beta} \right\}$ is a $\K$-basis of $F_A\ltimes_{\Gamma}
F_B$. Therefore, all $\eta_{\alpha,\beta,\theta,\mu}$ are zero and consequently all
$\kappa_{\alpha,\beta}$ are zero.

\end{proof}

\begin{proposition}
 \label{projective}
Suppose, that $\Gamma$ acts by automorphisms on $B$. Let $P$ be a $\Gamma$-graded projective $A$-module, and $N$ a
projective $B$-module. Then the module $M\ltimes_{\Gamma} N$ is a
$\Gamma$-graded
projective $A\ltimes_{\Gamma} B$-module.
\end{proposition}
\begin{proof}
 Since $P$ is a $\Gamma$-graded projective $A$-module, it is a direct
 summand of a free module $F_A$ over $A$. We denote the corresponding
 inclusion and projection $\Gamma$-graded $A$-module homomorphisms by
 $i_P$ and $\pi_P$ respectively. Analogously, since $N$ is a
 projective $B$-module, it is a direct summand of some free $B$-module
 $F_B$. We denote the respective inclusion and projective
 homomorphisms by $i_N$ and $\pi_N$. Then we have maps 
 \begin{align*}
  i\colon M\ltimes_{\Gamma} N & \to F_A\ltimes_{\Gamma} F_B\\
m\otimes n & \mapsto i_M(m)\otimes i_N(m)
 \end{align*}
 and 
 \begin{align*}
  \pi\colon F_A\ltimes_{\Gamma} F_B & \to M\ltimes_{\Gamma} N\\
  f\otimes g & \mapsto \pi_M(f)\otimes \pi_N(g) .
 \end{align*}
 The maps $i$ and $\pi$ are $\Gamma$-graded
 $A\ltimes_{\Gamma}B$-module homomorphism by
 Proposition~\ref{homomorphism}. It is obvious that $\pi\circ i =
 1_{M\ltimes_{\Gamma}N}$. Thus, $M\ltimes_{\Gamma} N$ is a direct
 summand of $F_A\ltimes_{\Gamma} F_B$. But by Theorem~\ref{free}  the
 module
 $F_A\ltimes_{\Gamma} F_B$ is a free $\Gamma$-graded $A\ltimes_{\Gamma} B$-module. 

\end{proof}
Let $N$ be a projective $B$-module. Then Proposition~\ref{projective} shows,
that the functor $-\ltimes_{\Gamma}N$ preserves projective resolutions.
Note, that it does not map in general a minimal projective resolution
into a minimal projective resolution. 

\begin{proposition}
 \label{radical}
 Suppose that $\Gamma$ acts by
 automorphisms on $B$. If $A_{\epsilon}\cong \K$ and $\Rad(B) = 0$,
 then 
$$
\rad(A\ltimes_{\Gamma} B) = \rad(A) \ltimes_{\Gamma} B.
$$
\end{proposition}
\begin{proof}
 Note, that $A_{\epsilon}\ltimes_{\Gamma} B$ is a subalgebra of
 $A\ltimes_{\Gamma}B$ and it is isomorphic to the usual tensor product
 of algebras $A_{\epsilon} \otimes B$. 
By Corollary~\ref{structureofradical}, we have
\begin{align*}
 \rad(A\ltimes_{\Gamma} B)& = \Rad(
 (A\ltimes_{\Gamma}B)_{\epsilon})
 \oplus \bigoplus_{\gamma\ne \epsilon} (A\ltimes_{\Gamma}B)_{\gamma}\\&=
 \Rad(A_{\epsilon} \otimes B) \oplus
 \left(\bigoplus_{\gamma\ne\epsilon}
 A_{\gamma}\right) \ltimes_{\Gamma} B\\
 & = \Rad(\K\otimes B)\oplus \rad(A)\ltimes_{\Gamma}B =
 \rad(A)\ltimes_{\Gamma}B.
\end{align*}
\end{proof}

\begin{corollary}
 \label{minimalmap}
 Suppose $A_{\epsilon} \cong \K$ and $\Rad(B)=0$. Let $\phi\colon
 M_1\to M_2$ be a homomorphism of $\Gamma$-graded $A$-modules.
 Suppose, that $\im(\phi)\subset \rad(A)M_2$. Then for any $B$-module
 $N$, we have
 $$
 \im(\phi\ltimes_{\Gamma} N)\subset  \rad(A\ltimes_{\Gamma}
 B)(M_2\ltimes_{\Gamma} N).
 $$
 \begin{proof}
  We have
  $$
  \im(\phi\ltimes_{\Gamma}N) = \im(\phi)\ltimes_{\Gamma} N \subset
  \rad(A)M_2 \ltimes_{\Gamma} N.
  $$
  From Proposition~\ref{radical} we have
  $$
  \rad(A\ltimes_{\Gamma} B) (M_2\ltimes_{\Gamma}N ) = (\rad(A)\ltimes
  B)(M_2\ltimes_{\Gamma}N).
  $$
  As $\Gamma$ acts by automorphism on $B$, for every $\gamma\in\Gamma$
  we have $B^{\gamma}= B$. Therefore
  $$
  (\rad(A)\ltimes
  B)(M_2\ltimes_{\Gamma}N) = \rad(A)M_2\ltimes_{\Gamma} BN =
  \rad(A)M_2\ltimes_{\Gamma} N.
  $$
 \end{proof}
\end{corollary}
\subsection{Strong idempotent ideals}
\label{strong-idempotent}
Let $A$ be an algebra and $I$ a two-sided ideal of $A$. In this section we give
an overview of results from~\cite{idempotent-ideals} concerning the
inclusion functor $A/I-mod \to A-mod$. 

We always have a map $\phi_{X,Y}^i\colon \Ext_{A/I}^I(X,Y) \to
\Ext_A^I(X,Y)$ for $i\ge 0$ and $X$, $Y$ in $A/I-mod$, induced by the
canonical isomorphism $\phi_{x,y}^0\colon \Hom_{A/I}(X,Y)\to
\Hom_A(X,Y)$. 
Analogously, there are maps  $\psi_{X,Y}^I\colon \Tor_i^A(Y,X)\to
\Tor_i^{A/I}(Y,X) $ for $i\ge 0$ and $X$ left and $Y$ right
$A/I$-module, induced by the canonical isomorphism
$Y\otimes_{A}X \cong Y\otimes_{A/I} X$. 
In~\cite[Proposition~1.2 and Proposition~1.3]{idempotent-ideals},  there was proved the equivalence of the
following properties for $I$ a two-sided ideal of an algebra
$A$ and $k$ a natural number:
 \begin{enumerate}
  \item $\phi_{X,Y}^i\colon \Ext_{A/I}^i(X,Y)\to \Ext_{A}^i(X,Y)$ is
  an isomorphism for all $X$, $Y$ in $A/I-mod$ and all $1\le i\le k$.
  \item $\Ext_{A}^i(A/I, Y)=0$ for all $A/I$-module $Y$ and all
  $1\le i\le k$.
  \item $\psi_{X,Y}^i\colon \Tor_i^{A/I}(X,Y) \to \Tor_i^A(X,Y)$ is an
  isomorphism for all $X$ in $(A/I)^{op}-mod$ and $Y$ in $A/I-mod$ and
  all $0\le i\le k$.
  \item $\Tor_i^A(A/I, Y) = 0$ for all $Y$ in $A/I$-mod and all
  $1\le i\le k$.
  \item if $Y$ is an $A/I$-module and $\cdots\to P_1 \to P_0\to Y\to 0$
  is a minimal projective resolution of $Y$ in $A-mod$, then 
  $$
  P_k/IP_k \to \cdots \to P_0/IP_0 \to Y\to 0 
  $$
 is the beginning of a minimal projective resolution of $Y$ in
 $A/I$-mod. 
 \end{enumerate} 
\begin{definition}
If one of the above conditions holds, we say that $I$ is a
\emph{$k$-idempotent ideal}. If the conditions hold for all $k\in \N$,
then we say that $I$ is \emph{strong idempotent ideal}. 
\end{definition}
We have the following obvious property of $k$-idempotent ideals
\begin{proposition}
 \label{chaine}
 If $$I_0\subset I_1\subset \dots \subset I_l\subset A$$ is a chain
 of ideals in $A$, such that for all $j$ the ideal $I_j/I_{j-1}$ of
 $A/I_{j-1}$ is $k$-idempotent, then $I_l$ is $k$-idempotent
 ideal of $A$
\end{proposition}
Note that from~ \cite[lemma~1.4(a) and
 Proposition~4.6]{idempotent-ideals}
 follows that
 an ideal $I$ is $1$-idempotent if and only if $I=AeA$ for some
 idempotent $e\in A$.

\begin{proposition}
 \label{twoidempotentcrit}
 Let $e\in A$ be an idempotent. An ideal $AeA$ is $2$-idempotent if
 and only if  the map induced by the multiplication in $A$ 
 $$
 Ae\otimes_{eAe} eA \to AeA
 $$
 is an isomorphism.
\end{proposition}
\begin{proof}
 It follows from \cite[lemma~1.4(b),
 and Proposition~4.6]{idempotent-ideals}.
\end{proof}
\begin{definition}
 Let $e\in A$ be an idempotent. An ideal $I=AeA$ is called an
 \emph{heredity} ideal if
 \begin{enumerate}
  \item $e\rad(A)e = 0$;
  \item $I$ considered as a left $A$-module is projective. 
 \end{enumerate}
\end{definition}
\begin{proposition}
 Suppose that $I=AeA$ and $I_A$ is projective as an $A$-module. Then $I$ is
 a strong idempotent ideal.
\end{proposition}
\begin{proof}
 It follows from the definition of strong idempotent ideal and
 \cite[Statement~3]{dlab-ringel}.
\end{proof}

\begin{corollary}
If $I$ is a heredity ideal, then $I$ is strong idempotent. 
\end{corollary}

\begin{proposition}
 An ideal $I=AeA$ is a heredity ideal if and only if 
 \begin{enumerate}
  \item $I$ is $2$-idempotent;
  \item $e\rad(A)e =0$.
 \end{enumerate}
\end{proposition}
\begin{proof}
 It follows from \cite[Statement~7]{dlab-ringel}.
\end{proof}

\begin{corollary}
 \label{2idempotent2strongly}
 If $I=AeA$ is $2$-idempotent, and $e\rad(A)e =0$, then $I$ is
 strong idempotent.  
\end{corollary} 
\subsection{Criterion of heredity}
\label{criterium}
In this section $\Gamma$ is always an ordered positive monoid    and
$A$ a $\Gamma$-graded algebra. Let $S$ be a $\Gamma$-set, that is a
set where $\Gamma$ acts by endomorphisms. 
Set $B=\Maps(S,\K)$. Then $B$ is a $\Gamma$-algebra and we can
consider the skew product algebra 
$$
C = A\ltimes_{\Gamma} B. 
$$
For simplicity, if $a\in A$ and $b\in B$ we will sometimes write $ab$ for the
element $a\otimes b$ of $C$.
For each subset $X$ of $S$ there is an idempotent in $C$
$$
\chi_{X} (x) = 
\begin{cases}
 1 & \mbox{if $x\in X$}\\
 0 & \mbox{if $x\notin X$.}
\end{cases}
$$
Define the algebra $C(X)$ by 
$$
C(X) = C/ C\chi_{\bar X} C, 
$$
where $\bar X$ denotes the complement of $X$ in $S$. 
In this section, we will prove some general results concerning the algebras
$C(X)$.

We say that
$\Gamma$ acts on effectively, if $\gamma_1 x \ne \gamma_2 x$ for all
$x\in S$ and all $\gamma_1\ne \gamma_2$ from $\Gamma$. 
From now on we will assume, that $\Gamma$ acts effectively on
$S$.
We introduce a partial order on $S$, by 
$$
x \le_{\Gamma} y \Leftrightarrow \exists\gamma\in \Gamma: y = \gamma x.
$$
\begin{definition}
We say that $X\subset S$ is convex if for all $x,y\in X$ it contains
all $z\in S$ which lie between $x$ and $y$.
\end{definition}
\begin{proposition}
 \label{split}
 Let $X$ be a convex subset of $S$. Then 
$$
C(X) \cong \chi_X C\chi_X.
$$
\end{proposition}
\begin{proof}
 Note, that we always have a surjective homomorphism
 \begin{align*}
\pi:&  \chi_X C \chi_X \to C(X)\\
 & a\otimes b  \mapsto [a\otimes b].
 \end{align*}
Let $x,y\in X$ and suppose that $y= \gamma x$ for some
$\gamma\in\Gamma$. Then 
$$
\chi_y C \chi_{\bar X} C \chi_x = \left\langle \chi_y a_1\chi_z a_2\chi_x | \gamma
=\gamma_2\gamma_1, z=\gamma_1 x, z\in \bar{X}, a_1\in A_{\gamma_1},
a_2\in A_{\gamma_2}\right\rangle.
$$
But if $z$ lies between $x$ and $y$, then $z\in X$ and the above set
is empty. Therefore the restriction of $\pi$ on $\chi_y C\chi_x $ is
injective, and since
$$
\chi_X C \chi_X = \bigoplus_{x,y\in X} \chi_y C \chi_x,
$$
the map $\pi$ is injective. 
\end{proof}
\begin{corollary}
 \label{case1}
 Let $J$ be a homogeneous basis of $A$ and $Y$ a $\Gamma$-convex
 subset of $S$. Then 
 $$
 \left\{ a\chi_y\middle| a \in J; \ y, \deg(a)y \in Y \right\}, 
 $$
 is a basis of $C(Y)$.
\end{corollary}
Let $X\subset Y$ be two subsets of $S$. Then we have a surjective
homomorphism
$$
C(Y) \to C(X).
$$
We give a criterion for projectivity of $C(X)$ over $C(Y)$. 

Suppose $A=A_1\otimes A_2$ is a product of two $\Gamma$-graded
algebras, that is 
$$
A_{\gamma} = \bigoplus_{\gamma=\gamma_1\gamma_2}
(A_1)_{\gamma_1}\otimes (A_2)_{\gamma_2}.
$$
For $i=1,2$, let $\Gamma_i$ be submonoid        of $\Gamma$ such that for
$\gamma\notin \Gamma_i$ the space $(A_i)_{\gamma}$ is zero. Denote by
$C_i$ the skew product $A_i\ltimes_{\Gamma_i} B$. 

\begin{theorem}
 \label{case2}
 Let $J_1$ and $J_2$ be homogeneous bases of $A_1$ and $A_2$,
 respectively. 
 Suppose that there are
 subsets $I_y$ of $J_2$ for each $y\in Y$, such that the set
 $$
 \left\{ a\otimes \chi_y | a \in I_y,\ y\in Y \right\}
 $$
 is a basis  of $C_2(Y)$ and the set
 $$
 \left\{ a_1a_2\otimes \chi_x | a_1\in J_1,\ a_2\in I_y,\ y\in Y,\
 \deg(a_1a_2)y\in Y  \right\}
 $$
 is a basis of $C(Y)$. Denote by $Z$ the difference $Y\setminus X$ and
 by $e$ the idempotent $\chi_Z$ of $C(Y)$. If 
 \begin{enumerate}
  \item  $\Gamma_1 X \cap Y = X$;
  \item  $\Gamma_2 Z \cap Y = Z$,
 \end{enumerate}
 then  
 \begin{enumerate} 
  \item $$
C(Y)e\otimes_{e C(Y)e} e C(Y) \cong  C(Y)e C(Y); 
$$  
  \item the set
  $$
   \left\{ a_1a_2\otimes \chi_x | a_1\in J_1,\ a_2\in I_x; x,
  \deg(a_2)x,\deg(a_1a_2)x\in X \right\}
  $$
  is a basis of $C(X)$. \end{enumerate}
 \end{theorem}
\begin{proof}

For $i=1,2$ denote by 
$C_i(Y)$ the algebras
$$
C_i/C_i\chi_{\bar Y}C_i.
$$
The proof of the theorem is a consequence of the following two
lemmata.
\begin{lemma}
 The set
$$
J= \left\{ a_1 a_2 \otimes \chi_y | a_1\in
J_1,\ a_2\in I_{y},\ y\in Y,\ \deg(a_1a_2)y\in Y, \ \deg(a_2)y\in Z  \right\},
$$
is a basis of 
$ C(Y)e C(Y) $.
\end{lemma}
\begin{proof}
 Note, that every element of $J$ is an element of $C(Y)eC(Y)$.
 In fact, denote $\deg(a_2)$ by $\gamma$, then
 $$
 a_1a_2\otimes \chi_y = a_1\otimes \chi_{\gamma y} \cdot a_2 \in
 C(Y)\chi_{\gamma y} C(Y) \subset C(Y) e C(Y). 
 $$
 Moreover, $J$ is a subset of a basis of $C(Y)$. Thus, it is enough to
 check that $J$ generates $C(Y)eC(Y)$. 
 It is clear that $C(Y)eC(Y)$ is generated by the set
 $$
 \tilde{J}= \left\{ a_1'a_2'\otimes \chi_z \cdot a_1a_2\otimes \chi_y
 \middle|
 \begin{array}{l} a_1,a_1'\in J_3;\\ \ a_2\in I_y,
  a_2'\in I_z;\ y\in Y\\
  \deg(a_1'a_2')z \in Y\\
  \deg(a_1a_2)y=z\in Z\end{array}\\
\right\}.
 $$
Let 
$a_1'a_2'\otimes \chi_z \cdot a_1a_2\otimes \chi_y
$ be an element of   $\tilde{J}$. Denote by $\gamma_i$ the degree of
$a_i$. Then $\gamma_1\gamma_2 y$ is an element of $Z$. From
the conditions of theorem it follows, that $z=\gamma_2 y$ is an element
of $Z$. In fact,  assume $z\in X$, then $\gamma_1 z\in X$ and
thus $\gamma_1\gamma_2 y\in X$. Contradiction. 

Now, the product
$$
a_1'a_2'\otimes \chi_z a_1 = a_1'a_2'a_1\otimes
\chi_{z}
$$
can be written as a linear combination of elements
$$
a_1''a_2''\otimes \chi_{z}
$$
with $a_1''\in J_1$,  $a_2''\in I_{z}$, since the set 
$$
 \left\{ a_1a_2\otimes \chi_x | a_1\in J_1,\ a_2\in I_y,\ y\in Y  \right\}
 $$
is a basis of $C(Y)$. And the
products
$$
a_2''\otimes \chi_{z} \cdot a_2 \otimes \chi_{x} = a_2''a_2\otimes
\chi_{x}
$$
can be written as linear combinations of $a_2'''\otimes
\chi_{x}$, where $a_2'''\in I_{x}$, since 
the set
 $$\left\{ a\otimes \chi_y | a\in I_y,\
y\in Y \right\}$$ is a basis of $C_1(Y)$. Moreover, in each case,
$$\deg(a_2''') y = \deg(a_2'')\deg(a_2)y = \deg(a_2'')z'\in \Gamma_2
Z\cap Y =  Z$$
and
$$
\deg(a_1'')\deg(a_2''')y = \deg(a_1'')\deg(a_2'')z' =
\deg(a_1'a_2')\deg(a_1)z'  = \deg(a_1'a_2')z \in Y.
$$
Thus every element of $\tilde{J}$ can be written as a linear
combination of elements from $J$. 
\end{proof}
\begin{lemma} The set   
 $$
 I= \left\{ (a_1 \otimes \chi_z) \otimes (a_2\otimes \chi_y) | a_1\in
J_1,\ a_2\in I_{y},\ \deg(a_2)y=z\in Z,\ \deg(a_1a_2)y\in Y  \right\}
 $$
is a basis of $C(Y)e\otimes_{eC(Y)e} eC(Y)$.
\end{lemma}
\begin{proof}
It is clear that all elements of $I$ are elements of
$C(Y)e\otimes_{eC(Y)e} eC(Y)$.
Since $\pi(I) = J$ is a basis of $C(Y)\chi_Z C(Y)$, it follows that
the elements of $I$ are linear independent. 
We show, that every element of $C(Y)e\otimes_{eC(Y)e} eC(Y)$ can be
written as a linear combination of elements from $I$. 

It is clear that $C(Y)e\otimes_{eC(Y)e} eC(Y)$ is generated by the set
$$
\tilde{I} = \left\{ (a_1'a_2'\otimes \chi_z)\otimes
(a_1a_2) \otimes \chi_y \middle|\begin{array}{c} a_1',a_1\in J_1; a_2'\in
I_z; a_2\in I_y;\\
y\in Y; z = \deg(a_1a_2)y \in Z\end{array} \right\}. 
$$
Now
$a_1a_2\otimes \chi_y  = a_1\otimes \chi_{y'}\cdot a_2\otimes y$ and
$y'\in Z$, because otherwise $$z= \deg(a_1)\deg(a_2) y = \deg(a_1) y'
\in \Gamma_1 X\cap Y = X,$$ that contradicts to the condition
$z\in Z$. Therefore $a_1\otimes \chi_{y'}= \chi_z\cdot a_1\cdot
\chi_{y'}$ is an element of $eC(Y)e$. Thus 
$$
(a_1'a_2'\otimes \chi_z)\otimes
(a_1a_2 \otimes \chi_y) = (a_1'a_2'a_1\otimes \chi_{y'} ) \otimes
(a_2\otimes \chi_y).
$$
Now $a_1'a_2'a_1\otimes \chi_{y'}$ can be written as a linear
combination of elements of the form $a_1''a_2''\otimes
\chi_{y'}$, with $a_1''\in J_1$ and $a_2''\in I_{y'}$. Since
$y'\in Z$, we have that $\deg(a_2'') y' \in \Gamma_2 Z \cap Y = Z$.
Therefore, $a_2''\chi_{y'} \in eC(Y) e$ and
$$
(a_1''a_2''\otimes \chi_{y'})\otimes (a_2\otimes \chi_y) =
(a_1''\otimes \chi_{y'}) \otimes (a_2''a_2\otimes \chi_y).
$$
Now, since $\{a_2\otimes \chi_y | a_2\in I_y,\ y\in Y\}$ is a basis of
$C_2(Y)$ , we can write
$a_2''a_2\otimes \chi_y$ as a linear combination of elements of the
form $a_2'''\otimes \chi_y$, where $a_2'''\in I_y$, moreover
$\deg(a_2''')y= \deg(a_2'')\deg(a_2)y= \deg(a_2'')\chi_{y'} \in
\Gamma_2 Z \cap Y = Z$. 
\end{proof}

From two previous lemmata it follows that $\pi$ gives a bijection
between the basis $I$ of $C(Y)e\otimes_{eC(Y)e} eC(Y)$ and the basis
$J$ of $C(Y)eC(Y)$. Therefore $\pi$ is an isomorphism.

The last claim of the theorem follows from the description of bases of
$C(Y)$ and $C(Y) e C(Y)$.

\end{proof}
\begin{corollary}
 \label{strange}
 Under the same conditions as in Theorem~\ref{case2}, suppose  $Z=\{z_1,z_2,\dots,z_m\}$ and
 \begin{enumerate}
  \item $i<j$, if $z_i \le_{\Gamma_2} z_j$;
  \item $j<i$, if $z_i \le_{\Gamma_1} z_j$.
 \end{enumerate}
 Denote by $X_k$ the set $X\cup\left\{ z_1,\dots, z_k \right\}$ and by
 $e_k$ the idempotent $\chi_{z_k}$ in $C(X_k)$. Then the ideal
 $C(X_k)e_kC(X_k)$ is $2$-idempotent.

 If additionally $\Rad(A_{\epsilon})=0$, then 
 $C(X_k)e_kC(X_k)$ are heredity ideals and the ideal
 $C(Y)\chi_{Z}C(Y)$ is strong idempotent. 
\end{corollary}
\begin{proof}
 We shall prove corollary by induction on $m$. For $m=1$ the claim follows
 from Theorem~\ref{case2}.

 Suppose we proved corollary for all $m\le N-1$. We shall prove it for
 $m=N$. Let us check that we can apply Theorem~\ref{case2} to the sets
 $X'=X\cup\left\{ z_1,\dots,z_{N-1} \right\}$ and $Z'=\left\{ z_N
 \right\}$. We have $\Gamma_2Z'\cap Y = Z'$. Suppose that this does
 not hold, then there exists $\gamma\in\Gamma$ such that
$\gamma z_N\in Y$ and $\gamma z_N\ne z_N$. Then, since $\Gamma_2Z\cap Y= Z$, we have $\gamma z_N =
 z_j$ for some $j<N$. But $z_N <_{\Gamma_2} \gamma z_N = z_j$, that
 should imply $N<j$. Contradiction. Thus $\Gamma_2 Z'\cap Y = Z'$. 

 Further $\Gamma_1 X'\cap Y = X'$. In fact, suppose there is
 $y\in X'$ and $\gamma\in \Gamma_1$ such that $\gamma y = z_N$. Since
 $\Gamma_1 X \cap Y = X$, we have $y\in Z$, that is $y = z_j$ for some
 $j\le N-1$. But then  $z_j=y <_{\Gamma_1} \gamma y = z_N $ and
 therefore $N<j$. Contradiction. 

 From Theorem~\ref{case2} we get
 \begin{enumerate}
  \item $C(X_N)e_NC(X_N) = C(Y)e_NC(Y)$ is a $2$-idempotent ideal;
  \item the algebra $C(X_{N-1})=C(X')$ has a basis
  $$
   J=\left\{ a_1a_2\otimes \chi_x | a_1\in J_1,\ a_2\in I_x; x,
   \deg(a_2)x,\deg(a_1a_2)x\in X_{N-1} \right\}.
  $$
\end{enumerate}
  For all $x\in X_{N-1}$ denote by ${I_x}'$ the set
  $$
  \left\{ a \middle| a \in I_x, \deg(a)x \ne z_N \right\}.
  $$
  Then
$$
J = \left\{ a_1a_2\otimes \chi_x \middle| a_1\in J_1, a_2\in {I_x}'; x,
\deg(a_1a_2)x\in X_{N-1}\right\}.
$$

We have $C_2(Y) e_N C_2(Y) = e_N C_2(Y)$. 
Therefore
$$
\left\{ a\chi_x \middle| a\in I_x; x,\deg(a)x\ne z_N \right\} =
\left\{ a\chi_x\middle| a\in {I_x}',\ x\in X_{N-1} \right\}.
$$
is a basis of 
 the algebra $C_2(X_{N-1})$.   
Denote by $Z''$ the set $\left\{ z_1,\dots, z_{N-1} \right\}$. 
Since $$\Gamma_2 Z''\cap X_{N-1} = \Gamma_2 Z'' \cap Y \cap
X_{N-1} = Z\cap X_{N-1} = Z'', $$
we can apply the claim of corollary to the set $X_{N-1} = X\coprod
Z''$ with $m=N-1$. 

Now, suppose that $\Rad(A_{\epsilon})=0$, then
$$
e_k \Rad(C(X_k))e_k  = \Rad(e_kC(X_k)e_k) = \Rad(e_kCe_k) =
\Rad(A_{\epsilon})=0,
$$
since $e_kCe_k = \chi_{z_k}C\chi_{z_k} \cong A_{\epsilon}$.
Now the claim follows from Corollary~\ref{2idempotent2strongly} and
Proposition~\ref{chaine}.
\end{proof}

\begin{theorem}
 \label{induction}
 Let $\Gamma$ be a \emph{commutative} positive monoid.
 Suppose that the $\Gamma$-graded algebra $A$ can be decomposed as the tensor product $A_1\otimes
 A_2 \otimes \dots A_m$, such that for all  $1\le i<j\le m$
$$
A_{ij} = A_i\otimes A_{i+1} \otimes\dots \otimes A_j
$$
is a $\Gamma$-graded subalgebra of $A$. 

Let 
$J_i$ be a $\Gamma$-homogeneous basis of $A_i$ for $1\le i \le m$,
and 
 $Y= \coprod_{i=1}^m Z_i$  a $\Gamma$-convex subset of $S$. Denote by $Y_{j}$ the set $Z_1\coprod\dots\coprod Z_j$ and by
$e_{j}$ the idempotent $\chi_{\bar{Y_{j}}}$.

If  there are submonoids $\Gamma_i$
of $\Gamma$ such that\\ 
\begin{enumerate}
 \item the subalgebra $A_i$ is $\Gamma_i$-graded, that is, for all
 $\gamma\in \Gamma\setminus  \Gamma_i$ the space $(A_i)_{\gamma}$ is
 zero;
 \item  $\Gamma_j Z_{i}\cap Y \subset
 \coprod\limits_{k=i}^{j} Z_{k}$ for $1\le i\le j\le m$;
 \item $\Gamma_i Y_j \cap Y = Y_j $ for $1\le i < j \le m$;
\end{enumerate}
then for each $1\le j\le
m-1$
the natural map 
$$
C(Y_{j+1})e_j\otimes_{e_{j}C(Y_{j+1})e_j}e_j C(Y_{j+1}) \to
C(Y_{j+1})e_{j} C(Y_{j+1})
$$
is an isomorphism of vector spaces. 

Moreover, for $1\le j \le m $ the set       
$$
\left\{ a_1 a_2 \dots a_m\chi_y\middle|
\begin{array}{l}
 a_k\in J_k,\ 1\le k\le m; y\in Y_j\\
 \deg(a_ka_{k+1}\dots a_m)y\in Y_j,\ j\le k\le m
\end{array}
\right\}
$$
is a basis of $C(Y_{j})$.
Suppose additionally that $\Rad(A_{\epsilon})=0$ and that on each
set $Z_j$, $j\ge 2$ there is an ordering $\le_j$, satisfying
\begin{enumerate}
 \item $z\le_j z'$ if $z\le_{\Gamma_i} z'$, for all $i\ge j$;
 \item $z\ge_j z'$ if $z\le_{\Gamma_i} z'$, for all $i< j$.
\end{enumerate}
Then the ideals $C(Y_{j})e_{j-1} C(Y_{j})$, for $j\ge 2$ are strong idempotent.

By Proposition~\ref{chaine} we have also that the ideal
$C(Y)\chi_{\bar{Z_1}}C(Y)$ of $C(Y)$ is strong idempotent.
\end{theorem}
\begin{proof}
 We prove the theorem by induction on $n$. 
 The case $m=1$ is proved in Corollary~\ref{case1}.

 Suppose the claim of the theorem holds for all  $m\le N$. We then prove 
 it             for $m=N+1$. 

 Decompose $Y$ as the disjoint union of $N$ sets $Y_2$,
 $Z_3$,..., $Z_{N+1}$.
 Denote by $\Gamma_{12}$ the submonoid of $\Gamma$ generated by
 $\Gamma_1$ and $\Gamma_2$. Then $A_{12}$ is a $\Gamma_{12}$-graded
 algebra. We claim that the conditions of the theorem are satisfied
 for the same set $Y$ and the same algebra $A$ and
 \begin{enumerate} 
  \item $m=N$;
  \item $Z'_1= Y_2$, $Z'_2 = Z_3$, \dots, $Z'_N = Z_{N+1}$;
  \item $A'_{1} =A_{12} $, $A'_2= A_3$,\dots, $A'_N = A_{N+1}$;
  \item $\Gamma'_1 = \Gamma_{12}$, $\Gamma'_2= \Gamma_3$, \dots,
  $\Gamma'_N = \Gamma_{N+1}$.
 \end{enumerate}
 In fact
 \begin{align*}
 A'_{ij} = A'_i\otimes A'_{i+1}\otimes\dots \otimes A'_j
=\begin{cases}
 A_{1,j+1} & \mbox{if $i=1$}\\
 A_{i+1,j+1} & \mbox{if $i\ne 1$}
\end{cases}
\end{align*}
are subalgebras of $A$ by the hypothesis of the theorem. Now, for
$j\ge i$ and $i\ne 1$
$$
\Gamma'_j Z'_i \cap Y= \Gamma_{j+1} Z_{i+1}\cap Y \subset
\coprod_{k=i+1}^{j+1} Z_k = \coprod_{k=i}^j Z'_k,
$$
$$
\Gamma'_i Y'_j\cap Y = \Gamma_{i+1} Y_{j+1}\cap Y = Y_{j+1} = Y'_j.
$$
For $i=1$ and $j> 1$ we have $j+1\ge 2$ and therefore
\begin{align*}\Gamma'_jZ'_1 \cap Y&= \Gamma_{j+1} Y_2\cap Y=
 \Gamma_{j+1}(Z_1\sqcup
Z_2)\cap Y\\ & \subset
(\Gamma_{j+1}Z_1\cap Y) \cup (\Gamma_{j+1}Z_2 \cap Y)\\ &\subset
\coprod_{k=1}^{j+1}Z_k
\cup \coprod_{k=2}^{j+1}Z_k = \coprod_{k=1}^j Z'_k.\end{align*}
 Note, that $\Gamma'_1Z'_1\cap Y\subset Z'_1$ is equivalent to
 $\Gamma'_1Y'_1 \cap Y\subset Y'_1$. Thus it is only needed to check that for
 $i=1$ and $j\ge 1$ we have $\Gamma'_1 Y'_j \cap Y = Y'_j$. Or
 in other words, that $\Gamma_{12} Y_{j+1} \cap Y = Y_{j+1}$.
 Note, that $\Gamma_1 Y_{j+1} \cap Y = Y_{j+1}$ and $\Gamma_2
 Y_{j+1}\cap Y = Y_{j+1}$ by the condition of the theorem. Now, we
 use $\Gamma$-convexity of $Y$. Let $\gamma\in \Gamma_{12}$. Then $\gamma$ can
 be written as a product $\gamma_1\gamma_2$ with
 $\gamma_{1}\in \Gamma_1$ and $\gamma_2\in \Gamma_2$. Let $y\in Y_{j+1}$ be
 such that $\gamma y \in Y$. Then we have 
 $$
 y \le_{\Gamma} \gamma_2 y \le_{\Gamma} \gamma_1\gamma_2 y = \gamma y.
 $$
 Since both $y$ and $\gamma y$ are elements of $Y$ and $Y$ is
 $\Gamma$-convex, the element $ \gamma_2 y$ lies
 in $Y$. Now
 $$
 \gamma_2y \in  \Gamma_2Y_{j+1} \cap Y 
  = Y_{j+1}, 
$$
$$
\gamma y  = \gamma_1(\gamma_2 y) \in \Gamma_1 Y_{j+1} \cap
Y = Y_{j+1}.
$$
 Thus
$\Gamma_{12} Y_{j+1} = Y_{j+1}$.

 Therefore, from the induction assumption for $2\le j \le N$ the
 natural map
 $$
 C(Y_{j+1})e_j\otimes_{e_{j}C(Y_{j+1})e_j}e_j C(Y_{j+1}) \to
C(Y_{j+1})e_{j} C(Y_{j+1})
 $$
 is an isomorphism. If the additional ordering assumptions are satisfied for
sets $Z_2$, $Z_3$, \dots,$Z_{N+1}$ and monoids $\Gamma_2$,
$\Gamma_3$,\dots, $\Gamma_{N+1}$ then
 they are automatically satisfied for sets $Z'_2 = Z_3$, $Z'_3 = Z_4$,
 \dots,$Z'_N = Z_{N+1}$ and monoids $\Gamma'_2 = \Gamma_3$, $\Gamma'_3
 = \Gamma_4$, \dots, $\Gamma'_N= \Gamma_{N+1}$. 
 Therefore, in that case, the ideals $C(Y_{j})e_{j-1} C(Y_{j})$ for $3\le j\le N$ are strong
 idempotent. 
 
 Returning to the general case, we now explore the consequences of the
 induction hypothesis on  bases. Since $\left\{ a_1a_2\middle| a_1\in J_1,
 a_2\in J_2 \right\}$ is a homogeneous basis of $A_{12}$, the sets    
 $$
 \left\{ a_1 a_2 \dots a_{N+1}\chi_y\middle|
\begin{array}{l}
 a_k\in J_k,\ 1\le k\le N+1\\
 \deg(a_ka_{k+1}\dots a_{N+1})y\in Y_j,\ j\le k\le N+1\\
 y \in Y_j
\end{array}
\right\}
$$
are  bases of
 $C(Y_j)$  for $3\le j \le N+1$, 
and the set    
$$
\left\{ (a_1 a_2)a_3 \dots a_{N+1}\chi_y\middle|
\begin{array}{l}
 a_k\in J_k,\ 1\le k\le {N+1}\\
 \ y \in Y_2\\
 \deg(a_ka_{k+1}\dots a_{N+1})y\in Y_2,\ 3\le k\le {N+1}\\
 \deg( (a_1a_2)a_3\dots a_{N+1})y \in Y_2
\end{array}
\right\}
$$
is a basis of $C(Y_2)$.   

Let $(a_1a_2)a_3\dots a_{N+1} \chi_y$ be an element of the last set.
Denote by $z$ the element $\deg(a_3\dots a_{N+1})y$. Then $z\in Y_2$.
We have
$$
z < \deg(a_2)z < \deg(a_1)\deg(a_2)z = \deg(a_1a_2)z.
$$
Since $z$ and $\deg(a_1a_2)z$ are elements of $Y_2$, which is a
subset of $Y$, and $Y$ is $\Gamma$-convex, it follows that
$\deg(a_2)z$ is an element of $Y$. Now
$$
\deg(a_2)z \in \Gamma_2Y_2\cap Y =(\Gamma_2Z_1\cap Y)\cup
(\Gamma_2Z_2\cap Y) \subset (Z_1\sqcup Z_2) \sqcup Z_2 = Y_2.
$$
Therefore the above basis of $C(Y_2)$ can be written as
$$
\left\{ (a_1 a_2)a_3 \dots a_{N+1}\chi_y\middle|
\begin{array}{l} a_k\in J_k,\ 1\le k\le {N+1}\\
 \ y \in Y_2\\
 \deg(a_ka_{k+1}\dots a_{N+1})y\in Y_2,\ 1\le k\le {N+1}\\
\end{array}
\right\}.
$$

Now, the algebra $A_{2,N}$ with the decomposition 
 $A_2\otimes A_3\otimes \dots A_{N+1}$, submonoids $\Gamma_j$,
 $2\le j\le N+1$ and decomposition $Y_2\coprod Z_3\coprod\dots\coprod Z_{N+1}$ of
 $Y $ satisfy the conditions of the theorem for $m=N$. By the induction
 hypothesis we get that the set    
$$
\left\{  a_2a_3 \dots a_{N+1}\chi_y\middle|
\begin{array}{l}
 a_k\in J_k,\ 2\le k\le {N+1}\\
  y \in Y_2\\
 \deg(a_ka_{k+1}\dots a_{N+1})y\in Y_2,\ 2\le k\le {N+1}
\end{array}
\right\}
$$
is a basis of
 $C_{2,N+1}(Y_2) = (A_{2,N+1}\ltimes_{\Gamma} B)(Y_2)$.   

For each $y\in Y_2$ denote by ${I_y}$ the set
$$
\left\{ a_2a_3\dots a_{N+1} \middle| \begin{array}{l}
 a_k\in J_k,\ 2\le k\le {N+1}\\
 \deg(a_ka_{k+1}\dots a_{N+1})y\in Y_2,\ 2\le k\le {N+1}
\end{array} \right\}.
$$

Then the set    
$$
\left\{ a_1 \tilde{a}\chi_y\middle|
\begin{array}{l}
 a_1\in J_1;\tilde{a} \in I_y\\
  y \in Y_2;
 \deg(a_1\tilde{a})y\in Y_2\\
\end{array}
\right\}
$$
is a basis of $C(Y_2)$   and the set
$$
\left\{  \tilde{a}\chi_y\middle|
\begin{array}{l}
\tilde{a} \in I_y;
  y \in Y_2
\end{array}
\right\}
$$
is a basis of     $C_{2,N+1}(Y_2)$.

Denote by $\Gamma_{2,N+1}$ the submonoid of $\Gamma$ generated by
$\Gamma_2$, ..., $\Gamma_{N+1}$. Then $A_{2,N+1}$ is
$\Gamma_{2,N+1}$-graded algebra. We have
$$
\Gamma_1Z_1 \cap Y_2 \subset \Gamma_1Z_1\cap Y = Z_1.
$$
Next we show that
 $\Gamma_{2,N+1} Z_2 \cap Y_2 = Z_2$.
Let $\gamma\in \Gamma_{2,N+1}$ and $z\in Z_2$ be such that $\gamma z
\in Y_2$. Since $\Gamma$ is commutative, we can write $\gamma$ as a product $\gamma_{N+1}\dots
\gamma_2$,
where each $\gamma_s$ belongs to some $Z_s$. Then
we have
$$
z \le_{\Gamma} \gamma_2 z
\le_{\Gamma}\gamma_{3}\gamma_2z\le_{\Gamma} \dots
\le_{\Gamma}\gamma z.
$$
Since both elements $z$ and $\gamma z$ lie in $Y$ and $Y$ is
$\Gamma$-convex it follows, that element $\gamma_s\dots \gamma_2 z$
belongs to $Y$. 
Now
$$
\gamma_2 z \in \Gamma_{2} Z_2\cap Y = Z_2,
$$
$$
\gamma_3 \gamma_2 z \in \Gamma_3 Z_2 \cap Y \subset Z_2 \sqcup Z_3,
$$
$$
\gamma_4 \gamma_3\gamma_2 z \in \Gamma_4(Z_2\sqcup Z_3)\cap Y \subset
Z_2\sqcup Z_3 \sqcup Z_4.
$$
Proceeding this way, we get 
$$
\gamma_s\dots \gamma_2 z \in \coprod_{k=2}^s Z_k.
$$
In particular
$$
\gamma z \in \coprod_{k=2}^{N+1} Z_k.
$$
But, since $z\in Y_2 = Z_1\sqcup Z_2$, this means that $\gamma z \in
Z_2$.
Thus
$$
\Gamma_{2,N+1}Z_2 \cap Y_2 = Z_2.
$$
Therefore, we can apply Theorem~\ref{case2} to the decomposition
$A=A_1\otimes A_{2,N+1}$
of $A$ and $Y_2= Z_1\sqcup Z_2$. Note that $e_1= \chi_{Z_2}\in C(Y_2)$. Therefore
$$
C(Y_2)e_1\otimes_{e_1C(Y_2)e_1} C(Y_2) \to C(Y_2) e_1 C(Y_2)
$$
is an isomorphism. Moreover, if $\Rad(A_{\epsilon})=0$, and there is an ordering $\le_2$ on
$Z_2$ such that 
\begin{enumerate}
 \item $z\le_{\Gamma_i}z'\Rightarrow z\le_2 z'$, for all $i\ge 2$;
 \item $z\le_{\Gamma_1} z'\Rightarrow z\ge_2 z'$,
\end{enumerate}
then 
$$
z\le_{\Gamma_{2,N}} z'\Rightarrow z\le_2 z'.
$$
Hence        , we can apply Corollary~\ref{strange} and get that the ideal
$C(Y_2)e_1 C(Y_2)$ is strong idempotent. 

Now, returning to the general case, the set  
$$
\left\{ a_1\tilde{a}\chi_y\middle| a_1\in J_1;\tilde{a}\in I_y;
y,\deg(\tilde{a})y,\deg(a_1\tilde{a})y\in Y_1\right\}=
$$
$$
=
\left\{ a_1 a_2 \dots a_{N+1}\chi_y\middle|
\begin{array}{l}
 a_k\in J_k,\ 1\le k\le {N+1}\\
 y,\deg(a_2\dots a_{N+1})y, \deg(a_1\dots a_{N+1})y \in Y_1\\
 \deg(a_ka_{k+1}\dots a_{N+1})y\in Y_2,\ 3\le k\le {N+1}\\
\end{array}
\right\}
$$
 is a basis of    $C(Z_1)=C(Y_1)$. 
Let $a_1\dots a_{N+1}\chi_y$ be an element of   the last set. We know,
that 
$$
\deg(a_k\dots a_{N+1})y \in Y_2, \mbox{ for $3\le k\le N+1$.}
$$
Assume, that 
$$
\deg(a_k\dots a_{N+1})y \in Z_2,
$$
for some $k$.
Then 
$$
\deg(a_2\dots a_{N+1})y = \deg(a_2\dots a_{k-1})\deg(a_k\dots
a_{N+1})y \in Z_1\cap (\Gamma_{2,N+1}Z_2).
$$
But
$$
Z_1\cap \left( \Gamma_{2,N+1}Z_2 \right)
= Z_1\cap Y_2 \cap
(\Gamma_{2,N+1}Z_2) = Z_1 \cap Z_2 =\phi. 
$$
Thus, 
$$
\deg(a_k\dots a_{N+1})y\in Z_1, \mbox{ for all $k$}
$$ and the set
$$
\left\{ a_1 a_2 \dots a_{N+1}\chi_y\middle|
\begin{array}{l}
 a_k\in J_k,\ 1\le k\le {N+1}\\
 y \in Y_1\\
 \deg(a_ka_{k+1}\dots a_{N+1})y\in Y_1,\ 1\le k\le {N+1}\\
\end{array}
\right\}
$$
is a basis of the algebra $C(Y_1)$. \end{proof}
\section{Application to Schur algebras}
\label{application}
In this section we  apply the technique developed in the
previous section to the problem of constructing (minimal) projective
resolutions for  simple modules over the Borel subalgebra
$S^+(n,r)$ of the Schur algebra $S(n,r)$. We start with a short
overview of Schur algebras.  

\subsection{Results on combinatorics}
\label{combinatorics}
We shall use the following combinatorial notions.
\begin{definition}
A \emph{partition} $\lambda$  of  $r$ is a sequence
$\lambda=(\lambda_1,\lambda_2,\dots)$ of non-negative weakly decreasing  integers
$\lambda_1\ge\lambda_2\ge\dots \ge 0$ such that  $\sum\lambda_i=r$. The set of all
partitions of  $r$ is denoted by $\Lambda^+(r)$. The $\lambda_i$ are the \emph{parts}
of the  partition.
If $\lambda_{n+1} = \lambda_{n+2} = \cdots = 0$, we say $\lambda$ has
\emph{length} at most $n$.
The set of all partitions of length at most  $n$ is denoted by  $\Lambda^+(n,r)$.

Dropping the  condition that the $\lambda_i$ are decreasing, we say that
$\lambda$ is a \emph{composition} of $r$. The set of all compositions
of $r$ is denoted by $\Lambda(r)$.  The set of all compositions of $r$ of
length at most $n$ is denoted by  $\Lambda(n,r)$.
\end{definition}

There is a  natural ordering on the set $\Lambda(r)$:
\begin{definition} (Dominance order) For  $\lambda,\mu\in\Lambda(r)$, 
we say that $\lambda$ \emph{dominates}  $\mu$ and write $\lambda\dominate\mu$
if 
\[
\sum\limits_{i=1}^j\lambda_i\ge\sum\limits_{i=1}^j\mu_i
\]
for all $j$.
\end{definition}

Restricting, we get the dominance order on $\Lambda(n,r)$. 
Now, let $\Z^n$ be the $n$-dimensional lattice over the ring of integer
numbers. We can consider $\Lambda(n,r)$ as a subset $\Z^n$. We extend
the dominance order from $\Lambda(n,r)$ to $\Z^n$ by the following
definition

\begin{definition} (Dominance order) For  $z,\bar{z}\in\Z^n$, 
 we say that $\bar{z}$ \emph{dominates}  $z$ and write
 $\bar{z}\dominate z $
if 
\[
\sum\limits_{i=1}^j\bar{z}_i\ge\sum\limits_{i=1}^j z_i
\]
for all $j$.
\end{definition}
Let $\Psi_n$ be the submonoid    of $\Z^n$ generated by the vectors
$v_i-v_{i+1}$, where $\left\{ v_i\mid 1\le i\le n \right\}$ is the
standard basis of $\Z^n$.
Then $\Psi_n$ acts effectively on $\Z^n$ by bijections if we define
$\gamma z:= \gamma +z$. The partial order $\le_{\Psi_n}$ introduced in
Section~\ref{criterium}, now becomes 
$$
z\le_{\Psi_n} \bar{z} \mbox{ iff } \bar{z} -z \in \Psi_n.
$$

\begin{proposition}
 \label{dominance}
 The dominance order on $\Z^n$ coincides with $<_{\Psi_n}$. 
\end{proposition}
\begin{proof}
 It is clear that for any $z \in \Z^n$ and any $i$ between $1$ and
 $n-1$:
 $$
 z + v_i - v_{i+1} \dominate z.
 $$ 
 Thus $\bar{z}>_{\Psi_n} z$ implies $\bar{z} \dominate z$. 

 Let $z$ and $\bar{z}\in \Z^n$ be such that $\bar{z}\dominate z$.
 Denote by $k_j$ the difference
$$
\sum\limits_{i=1}^j\bar{z}_i - \sum\limits_{i=1}^j z_i.
$$
Then for all $j$ the number $k_j$ is positive. It easy to see, that 
$$
\bar{z} - z = \sum_{j=1}^{n-1} k_j(v_j - v_{j+1}) \in \Psi_n.
$$
Thus $\bar{z} >_{\Psi_n} z$. 
\end{proof}

It is clear that $(r,0,\dots,0)$ and $(0,\dots,0,r)$ are the maximal
and minimal elements of $\Lambda(n,r)$ with respect to the dominance
order, respectively.    We denote by $\Lambda^{1}(n,r)$ the smallest
$\Psi_n$-convex subset in $\Z^n$ containing $(r,0,\dots,0)$ and
$(0,\dots,0,r)$. Then $\Lambda^1(n,r)$ contains $\Lambda(n,r)$ but
does not coincide with it, since $z\in \Lambda^1(n,r)$ can have negative
coordinates. 
\begin{proposition}
For all $z\in \Lambda^1(n,r)$ we have $z_1\ge 0$ and 
$$
\sum_{i=1}^n z_i = r.
$$ 
\end{proposition}
\begin{proof}
 Since $z$ dominates $(0,\dots,0,r)$ we have 
 $$
 \sum_{i=1}^j z_i \ge 0
 $$
 for all $j\in \left\{ 1,2,\dots, n-1 \right\}$. In particular,
 $z_1\ge 0$. Moreover,
 $$
 \sum_{i=1}^n  z_i \ge r.
 $$
 Since $z$ is dominated by $(r,0,\dots, 0)$ we have
 $$
 \sum_{i=1}^n z_i \le r.
 $$
 Thus the claim of proposition is proved.
\end{proof}

We denote by $\Lambda^k(n,r)$ the subset of $\Lambda^1(n,r)$ 
of all $z$ such that $z_i\ge 0$ for all $i \in \left\{ 1,2,\dots,
k
\right\}$. Then $\Lambda^l(n,r) \subset \Lambda^k(n,r)$ if $l\ge k$
and $\Lambda^n(n,r) = \Lambda(n,r)$. 
Denote by $M^k(n,r)$ the difference $\Lambda^{k-1}(n,r) \setminus
\Lambda^{k}(n,r)$. Then
$$
M^k(n,r) =\left\{ (z_1,\dots,z_n) \middle| z_1\ge 0,\dots,z_{k-1}\ge
0,z_k<0,\,z\in \Lambda^1(n,r) \right\}.
$$
Denote by $\Psi_{n,k}$ the submonoid of $\Psi_{n}$ generated by the
elements $v_i-v_k$ for all $i\in \left\{ 1,2,\dots,k-1 \right\}$.
Note, that $\Psi_{n,k} \cap \Psi_{n,l} = \left\{ 0 \right\}$, if
$k\ne l$. 
Further, let
$$
\Phi_{n,k} = \Psi_{n,2} + \Psi_{n,3} + \dots + \Psi_{n,k}
$$
and 
$$
\Theta_{n,k} = \Psi_{n,k+1} + \Psi_{n,k+2} + \dots + \Psi_{n,n}.
$$
Note, that $\Phi_{n,k} \cap \Theta_{n,k} = \{0\}$ for all $k\in \left\{
2,3,\dots,n \right\}$.

Next we will
introduce an  order~$\le_k$ on $M^k(n,r)$, satisfying 
\begin{enumerate}
 \item $z_1\le_{\Phi_{n,k}} z_2 \Rightarrow z_1\le_k z_2$;
 \item $z_1\le_{\Theta_{n,k}} z_2 \Rightarrow z_1 \ge_k z_2$
\end{enumerate}
for $z_1$, $z_2\in M^k(n,r)$. 
For this let $\le_{lex}$ denote the lexicographic order on $\Z^n$.
Define the map
\begin{align*}
 \phi^k\colon \Z^n & \to \Z^{n+1}\\
 (z_1,z_2,\dots,z_n) & \mapsto \left( \sum_{i\ne k}
 z_i,\sum_{i=k+1}^nz_i,z_1,z_2,\dots,z_{k-1},z_n,z_{n-1},\dots,z_{k+1} \right).
\end{align*}
Define the order $\le_k$ on $M^k(n,r)$ by
$$
z\le_k z'\equiv \phi^k(z)\le_{lex} \phi^k(z').
$$

\begin{proposition}
 \label{order}
 The order $\le_k$ satisfies the above stated properties:
\begin{enumerate}
 \item $z_1\le_{\Phi_{n,k}} z_2 \Rightarrow z_1\le_k z_2$;
 \item $z_1\le_{\Theta_{n,k}} z_2 \Rightarrow z_1 \ge_k z_2$.
\end{enumerate}
\end{proposition}
\begin{proof}
 Note, that $\Phi_{n,k}$ is generated by the vectors $v_i - v_j$, with
 $i<j\le k$. To prove the first property it is enough to check that
 for all $z\in M^k(n,r)$
 $$
 \phi^k(z)\le_{lex} \phi^k(z+v_i-v_j), \mbox{ if $i<j\le k$.}
 $$
 But $\phi^k$ is a linear map and $\le_{lex}$ is compatible with
 addition. Thus, it is enough to check that $\phi^k(v_i-v_j)
 \ge_{lex} 0$. For $j<k$ we have
 $$
 \phi^k(v_i-v_j) = (0,0,0,\dots,0,1,0,\dots,0,-1,0,\dots,0)
 \ge_{lex} 0.
 $$
 Further
 $$
 \phi^k(v_i-v_k) = (1,\dots) \ge_{lex} 0.
 $$
 The monoid $\Theta_{n,k}$ is generated by the vectors $v_i-v_j$ with
 $i<j$ and $k+1\le j\le n$. To prove the second part of the
 proposition it is enough to check that 
 $$
 \phi^k(v_i - v_j) \le_{lex} 0.
 $$
 For $i<k<j$ we have
 $$
 \phi^k(v_i-v_j) = (0,-1,\dots)\le_{lex} 0.
 $$
 For $k<i<j$ we get
 $$
 \phi^k(v_i-v_j) = (0,0,0,\dots,0,0,\dots,-1,\dots)\le_{lex} 0.
 $$
 Finally
 $$
 \phi^k(v_k-v_j) = (-1,\dots) \le_{lex} 0.
 $$
\end{proof}

Denote by $I(n,r)$ the subset of $\N^r$ consisting from the elements
$i= (i_1,i_2,\dots, i_r)$, such that $i_k \in \left\{ 1,2,\dots, n
\right\}$ for all $k$ between $1$ and $r$. 
\begin{definition}
We write $i\le j$ for $i,\,j\in I(n,r)$ if $i_\sigma\le j_\sigma$ for all $\sigma$,
$1\le \sigma\le r$.
\end{definition}
Denote by $\Sigma_r$ the permutation group on $\left\{ 1,2,\dots,r
\right\}$. The group $\Sigma_r$ acts on
 $I(n,r)$  by the rule
\[ i\pi=(i_{\pi(1)},\dots,i_{\pi(r)}) \qquad (i \in I,\,\pi \in
\Sigma_r).\]
Then we can extend the action of $\Sigma_r$ on $I(n,r)^2$ by 
$$
(i,j)\pi = (i\pi,j\pi). 
$$
We denote by $\Omega(n,r)$ the set 
$$
\left.\raisebox{0.5ex}{$\left\{ (i,j)\in I(n,r)\times I(n,r)\,\middle|\, i\le j
\right\}$}\!\middle/\raisebox{-0.5ex}{$\Sigma_r$}\right..
$$
\begin{definition} We say that a composition $\lambda=(\lambda_1,\dots,\lambda_n)$ is the \emph{weight} of
$i\in I(n,r)$, written  $\lambda = \wt(i)$, if
\[ \lambda_\nu=\left|\left\{ \rho\in\left\{ 1,2,\dots,r
\right\}\middle|i_\rho=\nu\right\}\right| \]
for all  $\nu\in\left\{ 1,2,\dots,n \right\}$.
\end{definition}
It is clear that if $i\le j$, then $\wt(i)\dominate \wt(j)$. For
$\lambda$,$\mu\in \Lambda(n,r)$ such that $\lambda\dominate\mu$ let 
$$
\Omega(\lambda,\mu) = \left.\raisebox{0.5ex}{$\left\{ (i,j)\in
I(n,r)\times I(n,r)\,\middle|\, i\le j,\,
\wt(i) = \lambda,\,\wt(j)= \mu
\right\}$}\!\middle/\raisebox{-0.5ex}{$\Sigma_r$}\right..
$$

We denote by $T(n,r)$ the set of upper-triangular $n\times n$-matrices
over $\N$ such that the sum of all its entries is $r$. Let 
$$
T(\lambda,\mu) = \left\{ K=(k_{\sigma\rho})_{\sigma,\rho=1}^n \middle|
\sum_{{\rho}={\sigma}}^n k_{{\sigma}{\rho}}  = \lambda_{\sigma},\ 
\sum_{{\sigma}=1}^{\rho} k_{{\sigma}{\rho}} = \mu_{\rho},\
{\sigma,\rho}\in \left\{ 1,2,\dots,n \right\}\right\}.
$$
\begin{proposition}
 \label{correlation}
For $i$,$j\in I(n,r) $ such that $i\le j$ and $\sigma$,$\rho\in
\left\{ 1,2,\dots,n \right\}$, set
$$
t(i,j)_{\sigma,\rho} =\#\left\{ \tau\in\left\{ 1,2,\dots,r
\right\}\middle| i_{\tau}=\sigma,\,j_{\tau}=\rho \right\}.
$$
Then the map
\begin{align*}
 t\colon  \left\{ (i,j)\in I(n,r)\times I(n,r)\middle|\, i\le
 j,\,\wt(i)=\lambda,\,\wt(j)=\mu \right\} &\to T(\lambda,\mu) \\
 (i,j)&\mapsto \left(t(i,j)_{\sigma\rho}\right)_{\sigma,\rho=1}^n
\end{align*} 
induces a bijection between $\Omega(\lambda,\mu)$ and
$T(\lambda,\mu)$.
\end{proposition}
\begin{proof}
 Since $i\le j$, we get that for $\sigma>\rho$ the number
 $t(i,j)_{\sigma,\rho}=0$. Moreover
 \begin{align*}
  \sum_{{\rho}={\sigma}}^n t(i,j)_{{\sigma}{\rho}} &= 
  \sum_{{\rho}={\sigma}}^n  \#\left\{ \tau\in\left\{ 1,2,\dots,r
  \right\}\middle| i_{\tau}=\sigma,\,j_{\tau}=\rho \right\}\\
 &=\#\left\{ \tau\in\left\{ 1,2,\dots,r
  \right\}\middle| i_{\tau}=\sigma \right\}\\
 &= \wt(i)_{\sigma} = \lambda_{\sigma}
 \end{align*}
and 
\begin{align*}
 \sum_{\sigma=1}^{\rho}t(i,j)_{{\sigma}{\rho}} &=
\sum_{\sigma=1}^{\rho} \#\left\{ \tau\in\left\{ 1,2,\dots,r
  \right\}\middle| i_{\tau}=\sigma,\,j_{\tau}=\rho \right\}\\
 &= \#\left\{ \tau\in\left\{ 1,2,\dots,r
  \right\}\middle|\,j_{\tau}=\rho \right\}\\
 &=\wt(j)_{\rho} = \mu_{\rho}.
\end{align*} 
Thus the image of $t$ lies in $T(\lambda,\mu)$. It is clear that
$t$ is $\Sigma_r$-invariant. Thus $t$ induces a map from
$\Omega(\lambda,\mu)$ to $T(\lambda,\mu)$. For $K=(k_{\sigma,\rho})\in
T(\lambda,\mu)$ we set
$$
i = (1^{\lambda_1},2^{\lambda_2},\dots,n^{\lambda_n})
$$
and 
$$
j =
(1^{k_{11}},2^{k_{12}},\dots,n^{k_{1n}},2^{k_{22}},\dots,n^{k_{nn}}).
$$
Then $i\le j$ and $t(i,j)_{\sigma,\rho}= k_{\sigma,\rho}$. Thus
$t$ is surjective. Moreover, each pair $\left( i',j' \right)$ of
multi-indices, such that $t(i',j')=K$, can be transformed to
$(i,j)$ with the appropriate element of $\Sigma_r$. Thus $t$ is injective.
\end{proof}
\begin{corollary}
 There is a bijection between $\Omega(n,r)$ and $T(n,r)$. 
\end{corollary}
\begin{proof}
 In fact, $\Omega(n,r) =
 \coprod_{\lambda\dominate\mu}\Omega(\lambda,\mu)$ and $T(n,r) =
 \coprod_{\lambda\dominate\mu}T(\lambda,\mu)$. By
 Proposition~\ref{correlation} there is a bijection between sets
 $\Omega(\lambda,\mu)$ and $T(\lambda,\mu)$ for all
 $\lambda\dominate\mu$.
\end{proof}
\subsection{The Schur algebra $S(n,r)$ and the Borel subalgebra
$S^+(n,r)$}
\label{schur}
In this section we follow~\cite{green,mt,s}.

Let $\K$ be an infinite field (of any characteristic) and  $V$ the 
natural module over $\GL_n(\K)$ with basis $\{v_1,\dots,v_n\}$. Then
there is a diagonal action of  $\GL_n(\K)$ on the $r$-fold tensor
product $V^{\otimes r}$. With respect to the basis
{$\{v_i=v_{i_1}\otimes\cdots\otimes v_{i_r}\colon i\in I(n,r)\}$}, this
action is given by the formula
\[gv_i=gv_{i_1}\otimes\dots\otimes gv_{i_r}.\]
We denote by $\tau_{n,r}\colon\GL_n(\K)= \GL(V)\to \End_{\K}(V^{\otimes r})$ the corresponding
representation of the group $\GL_n({\K})= \GL(V)$.
\begin{definition}[{\cite{green}}]
The Schur algebra
$S_{\K}(n,r)$ is the linear closure of the group
$\{\tau_{n,r}(g):g\in \GL_n({\K})\}$.
\end{definition}
Let us denote by $B^+_n({\K})$ the subgroup of $\GL_n(\K)$ consisting
of the upper triangular matrices.

\begin{definition}[{\cite{green2}}] The upper Borel subalgebra
 $S^+_{\K}(n,r)$ of the
 Schur algebra $S_{\K}(n,r)$ is the linear
 closure of the group $\left\{\tau_{n,r}(g) \middle| g\in B^+_n({\K})\right\}$.
\end{definition}

We denote by
$ e_{i,j} $ the linear transformation of $V^{\otimes r}$ whose matrix,
relative to the basis $\left\{\, v_i: i\in I(n,r) \,\right\}$ of
$V^{\otimes r}$, has $1$ in place $(i,j)$ and zeros elsewhere.

Define
$$
\xi_{i,j} = \sum_{\pi\in\Sigma_r} e_{i\pi,j\pi}.
$$

\begin{proposition}[{\cite[Thm.\ 2.2.6]{mt}}]
 The set 
 \[
 \left\{\xi_{i,j}\,\middle|\,\overline{(i,j)}\in \left(\raisebox{0.5ex}{$I(n,r)\times
 I(n,r)$}\!\middle/\raisebox{-0.5ex}{$\Sigma_r$}\right)\right\}
\]
is a $\K$-basis of $S(n,r)$.
\end{proposition}
The next statement was proved in~\cite[\S\S 3, 6]{green2}.
\begin{proposition}
\label{c1}
	 The algebra $S^+_{\K}(n,r)$ has a ${\K}$-basis $\left\{
	 \xi_{i,j}: \overline{(i,j)}\in \Omega(n,r) \right\}$.
\end{proposition}
\subsection{Universal enveloping algebra of $sl_n^+$ and Kostant form}
\label{kostant}
Denote by $sl_n^+$ the lie algebra of upper triangular nilpotent
matrices.
Let $\A_n(\C)$ be its universal enveloping algebra over $\C$. 
 We shall consider $sl^{+}_n$ with the standard basis $\left\{
e_{ij}\mid 1\le i<j\le n\right\}$. Then $\A_n(\C)$ is generated as an algebra by the
elements $e_{1,2},e_{2,3},\dots, e_{n-1,n} $. 

We order the elements $e_{ij}$ in such way, that $e_{ij}\le
e_{i'j'}$ if and only if $$(j,i)\ge_{lex} (j',i').$$ 
In other words, 
$$
e_{12} > e_{13} > e_{23} > \dots > e_{1k}>e_{2k}> \dots > e_{k-1,k} >
\dots > e_{n-1,n}.
$$
We always
assume that in the product $\prod\limits_{i<j}e_{ij}^{k_{ij}}$ the
generators increase from the left to right, with respect to the above
order. For
example, if $n=3$ and 
$$
(k_{ij})_{i,j=1}=
\left(
\begin{array}{ccc}
 0 & 1 & 1 \\
 0 & 0 & 2 \\
 0 & 0 & 0 
\end{array}
\right)
$$
then 
$$
\prod_{i<j} e_{ij}^{k_{ij}} = e_{23}^2 e_{13} e_{12}.
$$
It follows from the Poincare-Birkhoff-Witt Theorem, that 
the set 
$$
\mathbb{B}_n =\left\{ \prod_{1\le i<j\le n} e_{ij}^{k_{ij}}\middle| k_{ij}\in \N
\right\}
$$
is a $\C$-basis of $\A_n(\C)$.
 Denote by $e_{ij}^{(k)}$
the element $\frac{1}{k!}e_{ij}^{k}$ of the algebra $\A_n(\C)$.
We define $\A_n(\Z)$ to be the $Z$-sublattice of $\A_n(\C)$ generated by
the set 
$$
\overline{\mathbb{B}}_n =\left\{ \prod_{i<j} e_{ij}^{(k_{ij})}\middle| k_{ij}\in \N
\right\}.
$$
\begin{proposition}
 The set $\A_n(\Z)$ is a subring of $\A_n(\C)$. In other words,
 $\A_n(\Z)$  is a $\Z$-algebra. It is called the \emph{Kostant form}
 of the universal enveloping algebra $\A_n(\C)$ over $\Z$. 
\end{proposition}
\begin{proof}
 For a proof see~\cite[Lemma~2 after Proposition~3]{kostant}
 and~\cite[Remark~3]{kostant} thereafter. 
\end{proof} 

\begin{definition}
 For any field $\K$, the algebra $\A_n(\K) := \K\otimes_{\Z} \A_n(\Z)$ is called
 \emph{Kostant form} of the algebra $\A_n(\C)$ over 
 $\K$.
\end{definition}
Define a degree function on $\overline{\mathbb{B}}_n$, by 
\begin{align*}
 \deg\colon  \overline{\mathbb{B}}_n & \to \Psi_n\\
 \prod_{i<j} e_{ij}^{(k_{ij})} &\mapsto \sum_{i<j}
 k_{ij}(v_i-v_j).
\end{align*}
This makes $\A_n(\K)$ into a $\Psi_n$-graded algebra. 
We define subset $\overline{\mathbb{B}}_{n,k}$ of
$\overline{\mathbb{B}}_n$ by 
$$
\overline{\mathbb{B}}_{n,k} := \left\{ \prod_{i<k}
e_{ik}^{(l_i)}\middle|
l\in \N_0\right\}.
$$
\begin{remark}
 \label{tensorproductdecomposition}
Let $\A_{n,k}(\K)$ be the $\K$-vector subspace of $\A_n(\K)$
generated by $\overline{\mathbb{B}}_{n,k}$. 
Then
 $$
 \A_n(\K) \cong \A_{n,n}(\K) \otimes_{\K}
 \A_{n,n-1}(\K)\otimes_{\K} \dots \otimes_{\K} \A_{n,2}(\K)
 $$
 as $\K$-vector spaces,
 since 
 $$
 \overline{\mathbb{B}}_n = \overline{\mathbb{B}}_{n,n}\times
 \overline{\mathbb{B}}_{n,n-1}\times \dots \times
 \overline{\mathbb{B}}_{n,2}.
 $$
\end{remark}
Note that each $\A_{n,k}(\K)$ is graded over $\Psi_{n,k}$, since
$\deg(\overline{\mathbb{B}}_{n,k})\subset \Psi_{n,k}$. 
For $l>k$ denote by $\A_{n,l,k}(\K)$ the tensor product
$$
\A_{n,l}(\K)\otimes_{\K}\A_{n,l-1}(\K)\otimes_{\K}\dots
\otimes_{\K} \A_{n,k}(\K).
$$
It is a subspace of $\A_n(\K)$ with the basis
$$
\overline{\mathbb{B}}_{n,l}\times
 \overline{\mathbb{B}}_{n,l-1}\times \dots \times
 \overline{\mathbb{B}}_{n,k}.
$$
\begin{proposition}
 \label{subalgebra}
 For all $l>k$, the space $\A_{n,l,k}(\K)$ is a subalgebra of $\A_n(\K)$.
\end{proposition}
\begin{proof}
 Let $\frak{g}$ be a Lie subalgebra of $sl_n^+$. Then, from
 Poincare-Birkhoff-Witt Theorem, it follows, that
 $\frak{U}(\frak{g})$ is a subalgebra of $\A_n(\C)$. Note, that
 $$
 \frak{g}_{l,k} = \C\left\langle e_{ij} \middle| i<j,\, k\le j\le l \right\rangle
 $$
 is a Lie subalgebra of $sl_n^+$, since
 $$
 [e_{ij},e_{i'j'}] =
 \begin{cases}
  e_{ij'} & j=i'\\
 -e_{i'j} & i=j'\\
 0 &\mbox{otherwise.}
 \end{cases}
 $$
 
 Thus $\frak{U}(\frak{g}_{l,k})$ is a
 subalgebra of $\A_n(\C)$. But now, it follows from the definition
 that
 $$
 \A_{n,l,k}(\Z) = \frak{U}(\frak{g}_{l,k})\cap \A_n(\Z)
 $$
 and, therefore, $\A_{n,l,k}(\Z)$ is a subalgebra of $\A_n(\Z)$. Since
 the property ``to be subalgebra'' ``commutes'' with the functor
 $\K\otimes_{\Z}-$, it follows, that $\A_{n,l,k}(\K)$ is a subalgebra of
 $\A_n(\K)$. 
\end{proof}

\subsection{Main results}
\label{main}
Define $B_n$ to be the algebra of all $\K$-valued functions on
$\Z^n$. Then $\Psi_n$ acts on $B_n$ by shifts
$$
f^{\gamma}(z):= f(z-\gamma).
$$

We consider $\Lambda(n,r)$ as a subset of $\Z^n$. Denote by $\chi$ the
indicator function of $\Lambda(n,r)$, that is
$$
\chi(\gamma)=
\begin{cases}
 1, &\gamma\in \Lambda(n,r)\\
 0, & \gamma\notin \Lambda(n,r).
\end{cases}
$$

Then $\chi$ is an idempotent in the
algebra $B$, and therefore $e = 1_{\A_n(\K)}\otimes \chi$ is an idempotent in the algebra
$\Cf_n(\K):= \A_n(\K)\ltimes_{\Psi_n} B$. Denote by $\bar{e}$ the idempotent
$1-e$ in $\A_n(\K)\ltimes_{\Gamma} B$. 
\begin{theorem}
 \label{main1}
 The ideal $\Cf_n(\K)\bar{e}\,\Cf_n(\K)$ of $\Cf_n(\K)$ is strong
 idempotent and the set     
 $$
 \mathbb{J} =  \left\{ \prod_{i<j}e_{ij}^{(k_{ij})}\chi_{\mu} \middle|
 \lambda,\mu\in \Lambda(n,r),\,\lambda\dominate\mu,\, (k_{ij})_{i,j=1}^n \in T(\lambda,\mu) \right\}
 $$
  is a basis of 
$$
 \Cf_n(\Lambda(n,r)) =
 \left.\raisebox{0.5ex}{$\Cf_n(\K)$}\!\middle/\raisebox{-0.5ex}{$\Cf_n(\K)\bar{e}\,\Cf_n(\K)$}\right.
 $$
 
\end{theorem}
\begin{proof}
 We will         apply Theorem~\ref{induction} to the situation
 \begin{enumerate}
  \item $A = \A_n(\K)$;
  \item $m = n-1$;
  \item $A_1 = \A_{n,n}(\K)$, $A_2 = \A_{n,n-1}(\K)$, \dots,
  $A_{n-1} = \A_{n,1(\K)}$;
  \item $J_1 = \overline{\mathbb{B}}_{n,n}$, $J_2 =
  \overline{\mathbb{B}}_{n,n-1}$, \dots,
  $J_{n-1} = \overline{\mathbb{B}}_{n,1}$;
  \item $\Gamma_1 = \Psi_{n,n}$, $\Gamma_2 = \Psi_{n,n-1}$, \dots,
  $\Gamma_{n-1} = \Psi_{n,1}$;
  \item $Y = \Lambda^1(n,r)$;
  \item $Z_1 = \Lambda(n,r)$, $Z_2 = M^{n-1}(n,r)$, \dots,
  $Z_{n-1} = M^2(n,r)$ with the orderings $\le_k$ considered in~Proposition~\ref{order} on them.
 \end{enumerate}
 Let us check that the conditions of Theorem~\ref{induction} are
 satisfied:
 \begin{enumerate}
  \item The monoid $\Psi_n$ is commutative.
  \item By Proposition~\ref{subalgebra}, the subspaces
  $$
  A_{ij} = A_i\otimes A_{i+1}\otimes \dots \otimes A_j =
  \A_{n,l,k}(\K) \mbox{ for $l=n+1-i$ and $k=n+1-j$}
  $$
  are subalgebras of $A = \A_n(\K)$ for all $1\le i,j\le n$.
  \item The algebras $\A_{n,k}(\K)$ are $\Psi_{n,k}$-graded,
  since $\deg(\overline{\mathbb{B}}_{n,k})\subset \Psi_{n,k}$.
  \item Let $j\ge i$ and  $l=n+1-j$ and $k=n+1-i$. Then
  \begin{multline*}
   \Gamma_jZ_i \cap Y = \Psi_{n,l}M^k(n,r) \cap \Lambda^1(n,r) \\
    = \left\{z'\,\middle|\, z' =  z +
   \sum_{\tau=1}^{l-1}k_{\tau}(v_{\tau}-v_l);\,
   \begin{array}{l}z_1\ge 0,\dots,z_{k-1}\ge 0, z_k<0;\\ z,z'\in
    \Lambda^1(n,r); k_{\tau}\in \N_0\end{array}\right\}.
  \end{multline*}
  If $l= k$, then the last set coincides with $M^k(n,r)$ and
  therefore
  $$
  \Gamma_jZ_j\cap Y = M^k(n,r) = Z_j.
  $$
  If $l<k$, then 
  \begin{multline*}
   \left\{z'\,\middle|\, z' =  z +
   \sum_{\tau=1}^{l-1}k_{\tau}(v_{\tau}-v_l);\,
   \begin{array}{l}z_1\ge 0,\dots,z_{k-1}\ge 0, z_k<0;\\ z,z'\in
    \Lambda^1(n,r); k_{\tau}\in \N_0\end{array}\right\}=\\
    \begin{split}=&\left\{ z'\,\middle|\,\begin{array}{l}z'_1\ge0,
     \dots,z'_{l-1}\ge 0,z'_l<0;\\ z'_{l+1}\ge 0,\dots,z'_{k-1}\ge 0,z'_k<0\\
    z'\in \Lambda^1(n,r)\end{array} \right\}\\& \bigsqcup
\left\{ z'\,\middle|\,\begin{array}{l}z'_1\ge0,
     \dots,z'_{l-1}\ge 0,z'_l\ge 0;\\ z'_{l+1}\ge 0,\dots,z'_{k-1}\ge 0,z'_k<0\\
     z'\in \Lambda^1(n,r)\end{array} \right\}\end{split}\\
    \subset M^l(n,r)\sqcup M^k(n,r) = Z_i\sqcup Z_j \subset
    \coprod_{\tau=i}^{j} Z_{\tau}.
  \end{multline*}
  \item Let $l=n+1-j$. Then $Y_j = Z_1\sqcup Z_2 \sqcup \dots \sqcup Z_j =
  \Lambda^{l}(n,r)$. Hence     for $i<j$ and $k= n+1-i$
  \begin{align*}
   \Gamma_i Y_j \cap Y& = \Psi_{n,k} \Lambda^l(n,r)\cap
   \Lambda^1(n,r) \\
   &=\left\{ z' \middle| z'= z +
   \sum_{\tau=1}^{k-1}k_{\tau}(v_{\tau}-v_k);
   \begin{array}{l}z_1\le 0,z_2\le 0,\dots,
      z_l\le 0;\\ z,z'\in \Lambda^1(n,r); k_{\tau}\in \N_0\end{array}\right\} \\
   & = \Lambda^l(n,r)
  \end{align*}
  since $k>l$, and therefore the first $l$ coordinates of $z'$, which
  are obtained from the first $l$ coordinates of $z$
   upon addition of
  $\sum_{\tau=1}^{k-1}k_{\tau}(v_{\tau}-v_k)$, are positive.
  \item The radical of $A_0= \K$ is zero.
  \item The orders $\le_k$ on $M^k(n,r)$ satisfy the additional conditions of
  Proposition~\ref{order}.
 \end{enumerate}
Therefore an ideal $\Cf_n(\K)\bar{e}\,\Cf_n(\K)$ is strong
idempotent, and the set   
$$
\mathbb{I} = \left\{ a_n a_{n-1} \dots a_{2}\chi_{\mu}\middle|
\begin{array}{l}
 a_k\in \overline{\mathbb{B}}_{n,k},\ 2\le k\le n-1; \mu\in
 \Lambda(n,r)\\
 \mu+\deg(a_ka_{k-1}\dots a_2)\in \Lambda(n,r),\ 2\le k\le n 
\end{array}
\right\}.
$$
  is a basis of $\Cf_n(\K)(\Lambda(n,r))$.
  Let $a_n a_{n-1} \dots a_{2}\chi_{\mu}$ be an element of
  $\mathbb{I}$.
 Since $\deg(a_j)\in \Psi_{n,j}$, there are $k_{ij}\in \N_0$, such that 
$$
\deg(a_j) = \sum_{i=1}^{j-1}k_{ij}(v_i-v_j).
$$
Let 
$$
k_{jj} = \mu_j - \sum_{i=1}^{j-1}k_{ij}
$$
for  $j=  1,2,\dots,n$. Note that $k_{jj}$ is
the $j$-th coordinate of $z_j =\mu + \deg(a_j)+\deg(a_{j-1})+\dots+\deg(
a_2)$. Since $z_j\in \Lambda(n,r)$ it follows, that $k_{jj}\in \N_0$
for all $j$. Define $\lambda_i = \sum_{j=i}^n k_{ij}$. Then
$(k_{ij})_{i,j=1}^n \in T(\lambda,\mu)$. Thus $\mathbb{I} \subset
\mathbb{J}$. 

Now, let  $\prod_{i<j}e_{ij}^{(k_{ij})}\chi_{\mu} $ be an element of
$\mathbb{J}$. Then $(k_{ij})_{i,j=1}^n$ is an element of
$T(\lambda,\mu)$ for some $\lambda\dominate \mu$. We set 
$$
a_j = \prod_{i=1}^{j-1} e_{ij}^{(k_{ij})}
$$
for $j\in \left\{ 2,3,\dots,n \right\}$. Then $a_j$ is an element of 
$\overline{\mathbb{B}}_{n,j}$. Now, for all $k<j$, we have that the
$j$-th coordinate of $\mu + \deg(a_k)+\dots + \deg(a_2)$ is the same
as in $\mu$, and thus it is positive. For $k=j$, this $j$-th
coordinate is equal to
$$
\mu_j - \sum_{i=1}^{j-1}k_{ij} = \sum_{i=1}^j k_{ij} -
\sum_{i=1}^{j-1}k_{ij} = k_{jj}\ge 0.
$$
For $k\ge j$ the $j$-th coordinate of $\mu+ \deg(a_k)+ \dots +
\deg(a_2)$ is
$$
\mu_j - \sum_{i=1}^{j-1}k_{ij} + \sum_{i=j+1}^k k_{ji} =
\sum_{i=j}^k k_{ji} \ge 0.
$$
Therefore, for every $k$, the element $\mu + \deg(a_k) + \dots +
\deg(a_2)$ lies in $\Lambda(n,r)$. Therefore $\mathbb{J}\subset
\mathbb{I}$.
\end{proof}

\begin{theorem}
 \label{isomorphismprime}
The algebra
 $\Cf_n(\Lambda(n,r))$ is isomorphic to the Borel subalgebra $S^+(n,r)$ of the
 Schur algebra $S(n,r)$.
\end{theorem}

Before we prove the theorem let us explain how we can use it to
construct minimal projective resolutions of the simple
$S^+(n,r)$-modules.

Suppose, we
have constructed a  $\Psi_n$-graded (minimal) projective resolution
$P_{\bullet}$ of the trivial
module $M$ over the algebra $\A_n(\K)$. For each $\lambda \in
\Lambda(n,r)$ we denote by $N_{\lambda}$ the one dimensional
$B_n$-module, such that $fv = f(\lambda)v$ for all $v\in N_{\lambda}$ and all
$f\in B_n$. Then, by~Proposition~\ref{projective} and
Corollary~\ref{minimalmap}, $P_{\bullet}\ltimes_{\Psi_n}N_{\lambda}$ is  a
$\Psi_n$-graded (minimal) projective
resolution of the module $M\ltimes_{\Psi_n}N_{\lambda}$, since
$N_{\lambda}$ is a projective $B_{n}$-module, $\A_n(\K)_{0}= \K$ and
$\Rad(B_n) =0$.  Now,
the complex $$\Cf_n(\Lambda(n,r))\otimes_{\Cf_n(\K)}
\left(P_{\bullet}\ltimes_{\Psi_n}N_{\lambda}\right)$$ is a
 (minimal)
$\Psi_n$-\emph{graded} projective resolution
of the module $$\Cf_n(\Lambda(n,r))\otimes_{\Cf_n(\K)}
\left(M\ltimes_{\Psi_n}N_{\lambda}\right) \cong \K_{\lambda}.$$
Since $\Cf_n(\Lambda(n,r)) \cong S^+(n,r)$ is a finite dimensional
algebra, by Proposition~\ref{finitedimensional}, the complex
$$\Cf_n(\Lambda(n,r))\otimes_{\Cf_n(\K)}
\left(P_{\bullet}\ltimes_{\Psi_n}N_{\lambda}\right)$$ is  a (minimal) projective resolution of the
$S^+(n,r)$-module $\K_{\lambda}$ in the ungraded sense.   
\begin{proof}
In this proof we will use the notation introduced in the beginning of
Section~\ref{schur}. For simplicity we will write $u$ for the element
$u\otimes 1_B$ of $\Cf_n(\K)$.

Let $\left\{ e_k\mid 1\le k\le n \right\}$ be a basis of $V$. Denote
by $E_{ij}$ an endomorphism of $V$ given by
$E_{ij}(e_k) = \delta_{jk}e_i$. Define a representation $\rho_r$ of
$\A_n({\K})$ on $V^{\otimes r}$ by  
\begin{multline*}
\rho_r\left(e_{ij}^{(k)}\right)(v_1\otimes \dots \otimes v_r) =\\
=\sum_{\sigma_1<\sigma_2<\dots <\sigma_k} v_1\otimes
\dots\otimes E_{ij}(v_{\sigma_1})\otimes\dots \otimes
E_{ij}(v_{\sigma_2})\otimes \dots \otimes E_{ij}(v_{\sigma_k})\otimes \dots
\otimes v_r.
\end{multline*}
For $\lambda\in\Lambda(n,r)$, write     $\xi_{\lambda}$ for
$\xi_{i,i}$, for any $i\in I(n,r)$ such that $\wt(i)=\lambda$.
Extend $\rho_r$ to $\Cf_n(\K)$ by 
$$
\rho_r (a\otimes \chi_{\lambda})=
\begin{cases}
\rho_r(a) \xi_{\lambda} & \mbox{ if $\lambda\in \Lambda(n,r)$}\\
 0 & \mbox{ otherwise}.
\end{cases}
$$
It is clear that $\rho_r\left(\chi_{\overline{\Lambda(n,r)}}\right) = 0$. Therefore $\rho_r$ is a
representation of the algebra $\Cf_n(\Lambda(n,r))$. 
Note, that the image $\image(\rho_r)$ of $\rho_r$ is a subalgebra of
$\End(V^{\otimes r})$.

First we show that the image $\image(\tau_r)$ of
$\tau_r=\tau_{n,r}$ is a subset of
the image $\image(\rho_r)$.
The group $B^+_n(K)$ is generated by the elements of the form
$I+\mu E_{ij}$, where $I$ is the identity matrix, $\mu\in \K$ and
$i\le j$.
It is easy to check, that
$$
\tau_r(I + \mu E_{ii}) = \sum_{\lambda\in
\Lambda(n,r)}(1+\mu)^{\lambda_i} \xi_{\lambda} =
\sum_{\lambda\in\Lambda(n,r)}
(1+\mu)^{\lambda_i}\rho_r(\chi_{\lambda}).
$$
Suppose $i<j$. Then
\begin{align*}
 \tau_r(I+\mu E_{ij}) (v_1\otimes \dots \otimes v_r) &= 
 (v_1  + \mu E_{ij} v_1) \otimes \dots \otimes (v_r + \mu E_{ij}v_r)\\
 & = \sum_{k=0}^{r}\mu^k
 \rho_r\left(E_{ij}^{(k)}\right)(v_1\otimes \dots \otimes v_r).
\end{align*}
This shows, that $\image(\tau_r)\subset \image(\rho_r)$.

Thus $S^+(n,r)$ is
a subalgebra of $\image(\rho_r)$, and therefore it is a
subquotient of $\Cf_n(\Lambda(n,r))$. Now we use the fact that both
algebras have bases that are in bijection with the finite set $T(n,r)$
(see Proposition~\ref{c1} and Theorem~\ref{main1}). 
Therefore, they have equal dimensions and are       isomorphic.

\end{proof}

\bibliography{article}
\bibliographystyle{amsplain}
\end{document}